\documentclass[onecolumn,final]{elsart3p}

\usepackage{comment}
\usepackage{amsmath}       
\usepackage{amssymb}       
\usepackage{cite}              
\usepackage{chapterbib}    
\usepackage{color}
\usepackage{comment}
\usepackage{bm,bbm}
\usepackage{latexsym}
\usepackage{overpic}
\usepackage{times}
\usepackage{epsfig}
\usepackage{graphicx,graphics,rotating}
\usepackage{subfigure}
\usepackage{multicol}
\usepackage{multirow}

\usepackage{stackengine}

\newtheorem{theorem}{Theorem}[section]
\newtheorem{lemma}[theorem]{Lemma}
\newtheorem{corollary}[theorem]{Corollary}
\newtheorem{proposition}[theorem]{Proposition}

\newtheorem{remark}[theorem]{Remark}

\definecolor{MyDarkGreen}{rgb}{0,0.45,0}

\def\trait #1 #2 #3 {\vrule width #1pt height #2pt depth #3pt}
\def\fin{\hfill
        \trait .3 5 0
        \trait 5 .3 0
        \kern-5pt
        \trait 5 5 -4.7
        \trait 0.3 5 0
\medskip}

\newenvironment{proof}{\textit{Proof.}}{\fin}

\newcommand{\TERM}[1]{\big(\textbf{#1}\big)}

\newcommand{\INTP}{\footnotesize{I}}

\newcommand{\REAL}{\mathbbm{R}}

\newcommand{\restrict}[2]{{#1}{}_{|{#2}}}

\newcommand{\EOD}{\end{document}}

\renewcommand{\cv}{\mathbf{c}}

\newcommand{\fv}{\mathbf{f}}

\newcommand{\qv}{\mathbf{q}}

\newcommand{\uv}{\mathbf{u}}
\newcommand{\vv}{\mathbf{v}}
\newcommand{\wv}{\mathbf{w}}
\newcommand{\xv}{\mathbf{x}}
\newcommand{\yv}{\mathbf{y}}

\newcommand{\Vv}{\mathbf{V}}

\newcommand{\as}{a}
\newcommand{\bs}{b}

\newcommand{\ks}{k}

\newcommand{\ps}{p}
\newcommand{\qs}{q}

\renewcommand{\ss}{s}

\newcommand{\us}{u}
\newcommand{\vs}{v}

\newcommand{\xs}{x}
\newcommand{\ys}{y}

\newcommand{\Cs}{C}

\newcommand{\Qs}{Q}

\newcommand{\Ss}{S}

\newcommand{\Vs}{V}

\newcommand{\matA}{\mathsf{A}}
\newcommand{\matB}{\mathsf{B}}

\newcommand{\matQ}{\mathsf{Q}}

\newcommand{\calG}{\mathcal{G}}

\newcommand{\calM}{\mathcal{M}}

\newcommand{\HONE}  {H^1}
\newcommand{\HONEzr}{H^1_0}

\newcommand{\LTWO}  {L^2}

\newcommand{\LTWOzr}{L^2_0}

\newcommand{\HS}[1] {H^{#1}}

\newcommand{\CS}[1] {C^{#1}}

\newcommand{\PS}[1] {\mathbbm{P}_{#1}}

\newcommand{\Vshk}{\Vs^{h}_{k}}
\newcommand{\Vvhk}{\Vv^{h}_{k}}

\newcommand{\Vvhkt}{\TILDE{\Vv}^{h}_{k}}

\newcommand{\Qsh}  {\Qs^{\hh}}
\newcommand{\Qshk} {\Qs^{\hh}_{k}}
\newcommand{\Qshkk}{\Qs^{\hh}_{k-1}}
\newcommand{\Qshkkk}{\Qs^{\hh}_{\underline{k}}}

\newcommand{\Vvhpp}[1]{\Vv^{\FT,h}_{k}}

\newcommand{\SV}{\textit{SV}}

\newcommand{\FT}{\textit{F2}}

\newcommand{\VvhkSV}{\Vv^{\SV{},h}_{k}}

\renewcommand{\P} {\textrm{E}}            
\newcommand  {\E} {\textrm{e}}
\newcommand  {\V} {\textsf{v}}            
\newcommand  {\T} {\textsf{T}}            

\newcommand{\hh}{h}
\newcommand{\Th}{\Omega_{\hh}}

\newcommand{\xvP}{\xv_{\P}}        

\newcommand{\dims}{2}              

\newcommand{\hP}{\hh_{\P}}

\newcommand{\hE}{\hh_{\E}}

\newcommand{\mP}{\ABS{\P}}

\newcommand{\mE}{\ABS{\E}}

\newcommand{\Eset}{\mathcal{E}}    
\newcommand{\Vset}{\mathcal{V}}    

\newcommand{\NMB}{N}

\newcommand{\NPV}{\NMB^{\Vset}_{\P}}      
\newcommand{\NPE}{\NMB^{\Eset}_{\P}}      

\newcommand{\dV}{\,d\xv}
\newcommand{\dS}{\,ds}

\newcommand{\DIV}  {\text{div}\,}

\newcommand{\norE} {\mathbf{n}_{\E}}

\newcommand{\norPE}{\mathbf{n}_{\P,\E}}

\newcommand{\X}   {\mathbf{x}}

\newcommand{\xV}{\X_{\V}}

\newcommand{\TILDE}[1]{\widetilde{#1}}

\newcommand{\bil}[2]{\langle#1,#2\rangle}

\newcommand{\scal}  [2]{(#1,#2)}

\newcommand{\abs}   [1]{|#1|}

\newcommand{\ABS}   [1]{\left|#1\right|}

\newcommand{\snorm}  [2]{|#1|_{#2}}

\newcommand{\norm}  [2]{||#1||_{#2}}

\newcommand{\Norm}  [2]{\left|\!\left|#1\right|\!\right|_{#2}}

\newcommand{\HAT}[1]{\widehat{#1}}

\newcommand{\Pin}[1]{\Pi^{\nabla}_{#1}}
\newcommand{\Piz}[1]{\Pi^{0}_{#1}}
\newcommand{\PinP}[1]{\Pi^{\nabla,\P}_{#1}}
\newcommand{\PizP}[1]{\Pi^{0,\P}_{#1}}

\newcommand{\SPh}{\Ss_{\hh}}

\newcommand{\ash}{\as_{\hh}}
\newcommand{\bsh}{\bs_{\hh}}

\newcommand{\asPh}{\as^{\P}_{\hh}}
\newcommand{\bsPh}{\bs^{\P}_{\hh}}

\newcommand{\asP}{\as^{\P}}
\newcommand{\bsP}{\bs^{\P}}

\newcommand{\psh}{p_{\hh}}

\newcommand{\psI}{p_{\INTP}}

\newcommand{\qsh}{\qs_{\hh}}

\newcommand{\qsI} {\qs_{\INTP}}

\newcommand{\vsh} {\vs_{\hh}}

\newcommand{\uvh} {\uv_{\hh}}
\newcommand{\uvI} {\uv_{\INTP}}

\newcommand{\vvh} {\vv_{\hh}}
\newcommand{\vvI} {\vv_{\INTP}}

\newcommand{\vvht}{\TILDE{\vv}_{\hh}}

\newcommand{\wvh} {\wv_{\hh}}

\newcommand{\fvh} {\fv_{\hh}}

\newcommand{\kb}{\bar{k}}

\newcommand{\dvh}{{\bm\delta}_{\hh}}
\newcommand{\ssh}{\sigma_{\hh}}

\newcommand{\tbeta}{\widetilde{\beta}}
\newcommand{\zerov}{\mathbf{0}}

\usepackage{cite}

\begin{document}
%
\begin{frontmatter} 
  \title{A virtual element generalization on polygonal meshes of the
    Scott-Vogelius finite element method for the 2-D Stokes problem}

  \author[LANL] {G.~Manzini}
  \author[DICEA]{and A.~Mazzia}
  
  \address[LANL]{T-5 Applied Mathematics and Plasma Physics Group,
    Los Alamos National Laboratory,
    Los Alamos, NM 87545,
    USA}
  
  \address[DICEA]{Dipartimento di Ingegneria Civile, Edile e Ambientale - ICEA,
    Universit\`a di Padova,
    35131 Padova,
    Italy}
  
  \begin{abstract}
    The Virtual Element Method (VEM) is a Galerkin approximation
    method that extends the Finite Element Method (FEM) to polytopal
    meshes.
    In this paper, we present a conforming formulation that
    generalizes the Scott-Vogelius finite element method for the
    numerical approximation of the Stokes problem to polygonal meshes
    in the framework of the virtual element method.
    In particular, we consider a straightforward application of the
    virtual element approximation space for scalar elliptic problems
    to the vector case and approximate the pressure variable through
    discontinuous polynomials.
    We assess the effectiveness of the numerical approximation by
    investigating the convergence on a manufactured solution problem
    and a set of representative polygonal meshes.
    We numerically show that this formulation is convergent with
    optimal convergence rates except for the lowest-order case on
    triangular meshes, where the method coincides with the
    $\PS{1}-\PS{0}$ Scott-Vogelius scheme, and on square meshes, which
    are situations that are well-known to be unstable.
  \end{abstract}
  
  
\end{frontmatter}

\renewcommand{\arraystretch}{1.}
\raggedbottom



\section{Introduction}
\label{sec:intro}

Many physical phenomena can be described using the system of
incompressible Stokes equations as, for example, sedimentation and
bio-suspension processes~\cite{Hofer:2018}, droplet
dynamics~\cite{Kitahata-Yoshinaga-Nagai-Sumino:2013}, micro-fluidic
devices~\cite{Smith-Barbati-Santana-Gleghon-Kirby:2012}, fibrous
filter design~\cite{Linden-Cheng-Wiegmann:2018} and Stokes flows in
porous media~\cite{Bang-Lukkassen:1999}.
The Finite Element Method (FEM) was proven to be very successful in
the numerical treatment of the Stokes equations in variational form,
see
~\cite{Cai-Tong-Vassilevski-Wang:2010,Crouzeix-Raviart:1973,Girault-Raviart:1986}.
Although the FEM is very effective, it is also limited to a very
specific kind of unstructured meshes.
In fact, the Partial Differential Equations (PDEs) are discretized by
suitable polynomial trial and test functions that are conveniently
built only on unstructured meshes of triangular and quadrilateral
elements in two-dimensions (2-D) and tetrahedral and hexahedral
elements in three dimensions (3-D).
In the last decades, a great effort has been spent to overcome this
limitation and to develop numerical methods for steady state and
time-dependent problems working on more general polygonal and
polyhedral meshes,
~\cite{Wachspress:2015,
  Kuznetsov-Repin:2003,%
  Talischi-Paulino-Pereira-Menezes:2010,%
  Sukumar-Tabarraei:2004,%
  Tabarraei-Sukumar:2007,%
  BeiraodaVeiga-Lipnikov-Manzini:2014,
  Vacca-BeiraodaVeiga:2015}.

The Mimetic Finite Difference (MFD)
method~\cite{Lipnikov-Manzini-Shashkov:2014,BeiraodaVeiga-Lipnikov-Manzini:2014}
was among the first effective approaches to be proposed in this
direction.
In fact, this numerical approach uses only the degrees of freedom
without referring to any specific set of shape functions, and its
design allows the numerical model to preserve several fundamental
properties of the PDEs, such as maximum/minimum principles, solution
symmetries and the conservation of mass, momentum and energy,
see~\cite{Campbell-Shashkov:2001,Hyman-Shashkov:2001}.
Originally proposed for the numerical approximation of diffusion
problems~\cite{Brezzi-Lipnikov-Shashkov-Simoncini:2007,Brezzi-Lipnikov-Shashkov:2006},
the MFD method was then extended to convection–diffusion
problems~\cite{Cangiani-Manzini-Russo:2009}, and Stokes
equations~\cite{BeiraodaVeiga-Gyrya-Lipnikov-Manzini:2009,BeiraodaVeiga-Lipnikov:2010,BeiraodaVeiga-Lipnikov-Manzini:2010}.

The variational reformulation of the MFD method led to the Virtual
Element Method
(VEM)~\cite{BeiraodaVeiga-Brezzi-Cangiani-Manzini-Marini-Russo:2013}.
The VEM is a Galerkin method such as the finite element method where
the approximation space in every mesh element is composed by the
solutions of a differential problem.
The basis functions of such Galerkin formulation exist as specific
solutions of the elemental problems and are uniquely identified by a
special choice of the degrees of freedom.
However, they are ``virtual'' as they are never built explicitly, and
the bilinear forms of the variational formulation are approximated by
using special polynomial projections that are computable from the
degrees of freedom.

The VEM is strictly related to the finite element formulations on
polygonal and polyhedral
meshes~\cite{Manzini-Russo-Sukumar:2014,Cangiani-Manzini-Russo-Sukumar:2015,DiPietro-Droniou-Manzini:2018}.
Strong connections also exist with other discretization methods that
work on such kind of meshes, as for example the discontinuous
skeletal gradient
discretizations~\cite{DiPietro-Droniou-Manzini:2018}, and the Boundary
Element Method-based FEM (BEM-based
FEM)~\cite{Cangiani-Gyrya-Manzini-Sutton:2017:GBC:chbook}.

The first virtual element method was formulated as a conforming FEM
for the Poisson
problem~\cite{BeiraodaVeiga-Brezzi-Cangiani-Manzini-Marini-Russo:2013},
and then extended to convection-reaction-diffusion problems with
variable
coefficients~\cite{Ahmad-Alsaedi-Brezzi-Marini-Russo:2013,BeiraodaVeiga-Brezzi-Marini-Russo:2016b}.
Similarly, the nonconforming formulation was proposed for the Poisson
equation~\cite{AyusodeDios-Lipnikov-Manzini:2016},
and later extended to general
elliptic
problems~\cite{Cangiani-Manzini-Sutton:2017,Berrone-Borio-Manzini:2018},
Stokes problem~\cite{Cangiani-Gyrya-Manzini:2016}, eigenvalue
problems~\cite{Gardini-Manzini-Vacca:2019:M2AN:journal}, and the
biharmonic
equation~\cite{Antonietti-Manzini-Verani:2018,Zhao-Chen-Zhang:2016}.
The mixed virtual element method was proposed
in~\cite{Brezzi-Falk-Marini:2014}
and~\cite{BeiraodaVeiga-Brezzi-Marini-Russo:2016c} as the extension to
the virtual element setting of the Brezzi-Douglas-Marini and
Raviart-Thomas mixed methods~\cite{Boffi-Brezzi-Fortin:2013}.
The connection with the de~Rham diagrams and the Nedelec elements and
possible applications to electromagnetics have been explored
in~\cite{BeiraodaVeiga-Brezzi-Marini-Russo:2016a}.
A (surely non-exhaustive) list of other significant applications of
the VEM includes the works of References~\cite{%
Cangiani-Georgoulis-Pryer-Sutton:2016,%
BeiraodaVeiga-Lovadina-Vacca:2017,%
BeiraodaVeiga-Lovadina-Vacca:2018,%
BeiraodaVeiga-Chernov-Mascotto-Russo:2016,%
BeiraodaVeiga-Brezzi-Marini-Russo:2016a,%
BeiraodaVeiga-Brezzi-Marini-Russo:2016b,%
BeiraodaVeiga-Brezzi-Marini-Russo:2016c,%
BeiraodaVeiga-Brezzi-Marini-Russo:2016d,%
Benedetto-Berrone-Pieraccini-Scialo:2014,%
Berrone-Borio-Scialo:2016,%
Berrone-Pieraccini-Scialo:2016,%
Berrone-Benedetto-Borio:2016chapter,%
Berrone-Borio:2017,%
Perugia-Pietra-Russo:2016,%
Wriggers-Rust-Reddy:2016,BeiraodaVeiga-Lovadina-Mora:2015,%
BeiraodaVeiga-Manzini:2015,Mora-Rivera-Rodriguez:2015,%
Natarajan-Bordas-Ooi:2015,Berrone-Pieraccini-Scialo-Vicini:2015,%
Paulino-Gain:2015,%
Antonietti-BeiraodaVeiga-Mora-Verani:2014,%
BeiraodaVeiga-Brezzi-Marini-Russo:2014,%
BeiraodaVeiga-Brezzi-Marini-Russo:2014b,BeiraodaVeiga-Manzini:2014,%
BeiraodaVeiga-Brezzi-Marini:2013,Brezzi-Marini:2013,%
Berrone-Borio-Manzini:2018,
Benvenuti-Chiozzi-Manzini-Sukumar:2019,%
Antonietti-Manzini-Verani:2020,%
Certik-Gradini-Manzini-Mascotto-Vacca:2019,%
BeiraodaVeiga-Mora-Vacca:2019,%
Certik-Gardini-Manzini-Vacca:2018,%
BeiraodaVeiga-Manzini-Mascotto:2019,%
BeiraodaVeiga-Dassi-Vacca:2020,%
BeiraodaVeiga-Lipnikov-Manzini:2014}.

In this work, we are interested in extending the Scott-Vogelius finite
element method for the discretization of the 2D Stokes equation to the
virtual element setting.
The method that we present considers a discrete representation of the
two components of the velocity field by using the conforming virtual
element space originally proposed
in~\cite{BeiraodaVeiga-Brezzi-Cangiani-Manzini-Marini-Russo:2013,BeiraodaVeiga-Brezzi-Marini:2013}
and its modified (``enhanced'') version proposed
in~\cite{Ahmad-Alsaedi-Brezzi-Marini-Russo:2013}.
The scalar unknown, e.g., the pressure, is approximated by
discontinuous polynomials on the mesh elements.
The resulting discretization at the lowest order case on triangular
meshes coincides with the $\PS{1}-\PS{0}$ Scott-Vogelius FEM.
In such a case, the scheme can be non-convergent as the ``inf-sup''
stability condition cannot be ensured.
However, in all other cases our VEM provides a different
discretization, which performs well in the experiments we carried out.
The zero divergence constraint is satisfied in a variational sense,
i.e., the projection of the divergence on the subset of polynomials
used in the scheme formulation is zero.
It is worth mentioning that other virtual element approaches were
recently proposed in the literature that approximate the Stokes
velocity in such a way that its divergence is a polynomial that is set
to zero in the scheme.
This strategy provides an approximation of the Stokes velocity that
satisfies the zero divergence constraint in a pointwise sense.
We refer the interested reader to the works of References~\cite{
  BeiraodaVeiga-Lovadina-Vacca:2017,
  BeiraodaVeiga-Lovadina-Vacca:2018,
  BeiraodaVeiga-Mora-Vacca:2019,
  BeiraodaVeiga-Dassi-Vacca:2020,
  Chernov-Marcati-Mascotto:2021}.
However, the polynomial projection of the velocity divergence in our
VEM is zero up to the machine precision.
If we consider such projection as our numerical approximation to the
velocity divergence, such approximation is identically zero in the
computational domain.

\medskip
The paper is organized as follows.
The Stokes equations in strong and variational form are introduced in
Section~\ref{sec:Stokes}.
The virtual element method and its connection with the Scott-Vogelius
finite element method are presented and discussed in
Section~\ref{sec:VEM}.
The convergence behavior of our VEM is established in
Section~\ref{sec4:convergence}, where we derive an error estimate for
both velocity and pressure in the energy norm, and investigated
numerically in Section~\ref{sec:numerical} through a manufactured
solution test case that is solved on a set of representative polygonal
meshes.
This set of meshes includes also the triangular meshes where the
low-order original Scott-Vogelius finite element method may show a
well-known unstable behavior, thus motivating us to study the
``inf-sup'' stability of the method numerically.
Final remarks and hints for future work are given in
Section~\ref{sec:conclusions}.

\section{The Stokes problem and the virtual element discretization}
\label{sec:Stokes}

\subsection{Notation and technicalities}
\label{subsec:notation}
Consider the integer number $k>0$ and let $\omega$ be a bounded, open,
connected subset of $\REAL^{2}$.
According to the notation and definitions given
in~\cite{Adams-Fournier:2003}, $\LTWO(\omega)$ is the linear space of
square integrable functions defined on $\omega$ and $\HS{k}(\omega)$,
is the linear subspace of functions in $\LTWO(\omega)$ whose weak
derivatives of order less than or equal to $k$ are also in
$\LTWO(\omega)$.
We denote the norm and seminorm in $\HS{k}(\omega)$ by
$\norm{\cdot}{k,\omega}$ and $\snorm{\cdot}{k,\omega}$, respectively.
Throughout the paper, we prefer denoting the scalar product between
scalar and vector-valued fields by the integral notation, even if
sometimes we use the notation ``$(\cdot,\cdot)$'' for the sake of
conciseness.

In the formulation of the Stokes problem on the computational domain
$\Omega$, we use the linear space
\begin{align}
  \LTWOzr(\Omega):=\Big\{\qs\in\LTWO(\Omega)\,:\,\int_{\Omega}\qs\dV=0\Big\},
\end{align}
which is clearly a subspace of $\LTWO(\Omega)$.
This subspace is isomorphic to the quotient space
$\LTWO(\Omega)\backslash{\REAL}$, where two square integrable
functions are equivalent and identified as members of the same
equivalence class (which we still call ``function'') if their
difference is constant.

\subsection{Strong and weak form of the Stokes problem}

We are interested in the numerical discretization of the system of
incompressible Stokes equations
\begin{align}
  - \Delta\uv +\nabla\ps &= \fv \phantom{0}    \quad\textrm{in~}\Omega,\label{eq:stokes:A}\\[0.2em]
  \DIV\uv                &= 0   \phantom{\fv}  \quad\textrm{in~}\Omega,\label{eq:stokes:B}
\end{align}
for the velocity vector $\uv$ and the pressure $\ps$ defined on the
computational domain $\Omega$, which we assume to be an open, bounded,
polygonal subset of $\REAL^2$, whose boundary is denoted by $\Gamma$.
For the mathematical well-posedness of this problem, we consider the
homogeneous Dirichlet boundary condition
\begin{align}
  \uv = 0   \phantom{\fv}  \quad\textrm{on~}\Gamma.\label{eq:stokes:C}
\end{align}
Assumption~\eqref{eq:stokes:C} makes the exposition simpler and
immediate and is not restrictive on the design of the scheme as more general
boundary conditions can be included with some additional technicalities.
In particular, including nonhomogeneous Dirichlet boundary conditions
is straightforward and this case will be considered in the numerical
experiments.

\medskip
\noindent
To write the variational formulation of problem
\eqref{eq:stokes:A}-\eqref{eq:stokes:C}, we introduce the bilinear
forms
$$\as(\cdot,\cdot):\big[\HONE(\Omega)\big]^2\times\big[\HONE(\Omega)\big]^2\to\REAL$$
and
$$\bs(\cdot,\cdot):\big[\HONE(\Omega)\big]^2\times\LTWO(\Omega)\to\REAL,$$
which are defined as:
\begin{align}
  \as(\vv,\wv)&:=\int_{\Omega}\nabla\vv:\nabla\wv\dV
  \phantom{\int_{\Omega}\DIV\vv\,\qs\dV}\hspace{-1.25cm}
  \forall\vv,\wv\in\big[\HONE(\Omega)\big]^2,
  \label{eq:as:def}\\[0.5em]
  \bs(\vv,\qs)&:=-\int_{\Omega}\qs\DIV\vv\dV
  \phantom{\int_{\Omega}\nabla\vv:\nabla\wv\dV}\hspace{-1.25cm}
  \forall\vv\in\big[\HONE(\Omega)\big]^2,\,\qs\in\LTWO(\Omega).
  \label{eq:bs:def}
\end{align}
The symbol ``:'' in \eqref{eq:as:def} is the standard ``dot'' product
between two-dimensional tensors.
Now, the variational formulation is given by:
\emph{Find
  $(\uv,\ps)\in\big[\HONEzr(\Omega)\big]^2\times\LTWOzr(\Omega)$} such
that
\begin{align}
  \as(\uv,\vv) + \bs(\vv,\ps) &= (\fv,\vv) \phantom{0}        \qquad\forall\vv\in\big[\HONEzr(\Omega)\big]^2,\label{eq:stokes:var:A}\\[0.5em]
  \bs(\uv,\qs)                &= 0         \phantom{(\fv,\vv)}\qquad\forall\qs\in\LTWOzr(\Omega).            \label{eq:stokes:var:B}
\end{align}
The existence and uniqueness of the solution pair $(\uv,\ps)$,
cf.~\cite{Boffi-Brezzi-Fortin:2013,Girault-Raviart:1986,Girault-Raviart:1979},
follow on noting that the bilinear form $\as(\cdot,\cdot)$ is
continuous and coercive, and the bilinear form $\bs(\cdot,\cdot)$ is
continuous and satisfies the inf-sup condition:
\begin{align}
  \label{eq:inf-sup:continuous:setting}
  \inf_{\qs\in\LTWOzr(\Omega)\backslash{\REAL}}\sup_{\vv\in[\HONEzr(\Omega)\backslash{\REAL}]^2}\frac{ \bs(\vv,\qs) }{ \norm{\vv}{1,\Omega}\,\norm{\qs}{0,\Omega} }\geq\HAT{\beta},
\end{align}
for some real, strictly positive constant $\HAT{\beta}$.

We consider the two finite-dimensional approximation spaces $\Vvhk$
and $\Qshkkk$ for the vector and the scalar unknowns, where $k$ and
$\underline{k}$ are two integer numbers such that $\underline{k}\le
k-1$.
These spaces are labeled by $\hh$ to indicate that they are built on a
given mesh $\Th$. 
We assume that $\Vvhk$ is a conforming subspace of
$\big[\HONEzr(\Omega)\big]^2$ and $\Qshkkk$ a discontinuous subspace
of $\LTWOzr(\Omega)$.
We search for a vector field $\uvh\in\Vvhk$ and a scalar field
$\psh\in\Qshkkk$ that approximate $\uv$ and $\ps$, respectively.
These fields are the solution of the following variational
problem:
\emph{Find
  $(\uvh,\psh)\in\Vvhk\times\Qshkkk$
  such that
  }
\begin{align}
  \ash(\uvh,\vvh) + \bsh(\vvh,\psh) &= \bil{\fvh}{\vvh} \phantom{0}               \qquad\forall\vvh\in\Vvhk,  \label{eq:stokes:vem:A}\\[0.5em]
  \bsh(\uvh,\qsh)                   &= 0                \phantom{\bil{\fvh}{\vvh}}\qquad\forall\qsh\in\Qshkkk.\label{eq:stokes:vem:B}
\end{align}
In equations~\eqref{eq:stokes:vem:A} and~\eqref{eq:stokes:vem:B} we
use the virtual element approximation of the bilinear forms
$\as(\cdot,\cdot)$ and $\bs(\cdot,\cdot)$, which are denoted by
$\ash(\cdot,\cdot):\Vvhk\times\Vvhk\to\REAL$ and
$\bsh(\cdot,\cdot):\Vvhk\times\Qshkkk\to\REAL$, respectively.
Similarly, in the right-hand side of equation~\eqref{eq:stokes:vem:A}
we use the virtual element approximation of the right-hand side
of~\eqref{eq:stokes:var:A}, here denoted by $\bil{\fvh}{\cdot}$, where
$\fvh$ is assumed to be an element of the dual space of $\Vvhk$.

\section{The virtual element generalization of the Scott-Vogelius finite element method}
\label{sec:VEM}

\subsection{Mesh definition and regularity assumptions}
\label{subsec:mesh:regularity:assumptions}
Let $\mathcal{T}=\{\Th\}_{\hh}$ be a mesh family for $\Omega$, where
every mesh $\Th$ is a finite set of (closed) polygonal elements $\P$
such that $\overline{\Omega}=\cup_{\P\in\Th}\P$.
Each mesh is labeled by the subindex $\hh$, which is, as usual, the
maximum of the diameters of the mesh elements, i.e.,
$\hP=\sup_{\xv,\yv\in\P}\abs{\xv-\yv}$.
We assume that all the elements of a given mesh are nonoverlapping in
the sense that the intersection of the closure in $\REAL^2$ of any
pair of them can only be a mesh vertex or a mesh edge.
Therefore, the area of such intersection is zero.
Every polygonal element $\P$ has $\NPV$ vertices with coordinates
$\xV=(\xs_{\V},\ys_{\V})$.
These vertices are connected by $\NPE$ nonintersecting straight edges
$\E$, which form the boundary $\partial\P$.
The measure of $\P$ is denoted by $\mP$, its centroid (e.g., the
barycenter) by $\xvP:=(\xs_{\P},\ys_{\P})$, the unit outward vector
orthogonal to the elemental edge $\E$ by $\norPE$.
We also introduce the unit vector $\norE$, which is orthogonal to the
edge $\E$ and whose orientation is independent of the elements $\P$,
but it is fixed \emph{once and for all} in the mesh.
Note that for any given edge $\E$, the vector $\norPE$ may differ from
$\norE$ only by the multiplicative factor $-1$.

The mesh sequence used in the formulation of the method is required to
satisfy the two following conditions: there exists a positive constant
$\varrho$ such that
\begin{itemize}
\item 
  \textbf{(M1)} all polygonal elements $\P$ are star-shaped with
  respect to a disk with a radius, $r$, such that $r\ge\varrho\hP$;
\item
  \textbf{(M2)} all polygonal edges $\E\in\partial\P$ of all polygonal
  elements $\P$ satisfy $\hE\geq\varrho\hP$, where $\hE$ is the edge
  length.
\end{itemize}
We assume that $\varrho$ is independent of $\hh$, so that these
regularity assumptions are uniformly satisfied by all the meshes $\Th$
of the mesh family $\mathcal{T}=\{\Th\}_{\hh}$ used in the VEM
formulation.
Property \textbf{(M1)} implies that all the mesh elements are
\emph{simply connected} subsets of $\REAL^{2}$.
Property \textbf{(M2)} implies that the number of edges in the
elemental boundaries is uniformly bounded over the whole mesh family
$\mathcal{T}$.
We remark that under these conditions the theory of polynomial
approximation of functions in Sobolev spaces~\cite{Brenner-Scott:1994}
provides interpolation and projection error estimates that can be used
in the analysis of the method, see, e.g.,
Section~\ref{subsec:preliminary:results}.

\subsection{Scalar and vector approximation spaces}

If $\P$ is a generic mesh element and $\ell$ a nonnegative integer
number, we let $\PS{\ell}(\P)$ denote the linear space of polynomials
of degree up to $\ell$ defined on $\P$, with the useful convention
that $\PS{-1}(\P)=\{0\}$, and
$\big[\PS{\ell}(\P)\big]^2$ denotes the space of two-dimensional
vector polynomials of degree up to $\ell$ on $\P$.
We use the notation $\PS{\ell}(\Th)$ for the elementwise polynomials
of degree $\ell$ defined on the mesh $\Th$,
so that $\qs\in\PS{\ell}(\Th)$ is such that
$\restrict{\qs}{\P}\in\PS{\ell}(\P)$ for all $\P\in\Th$.
Similarly, if $\E$ is a generic mesh edge, we let $\PS{\ell}(\E)$
denote the linear space of polynomials of degree up to $\ell$ defined
on $\E$.

For the approximation of the scalar unknown $\ps$, we consider the
functional space of
discontinuous polynomials of degree $\underline{k}$ having zero
average on the mesh $\Th$, which we formally define as
$\Qshkkk:=\PS{\underline{k}}(\Th)\cap\LTWOzr(\Omega)$.

Instead, for the approximation of the vector unknown $\uv$ we consider
the functional space of vector-valued functions
\begin{align}
  \Vvhk:=\Big\{\vvh\in\big[\HONEzr(\Omega)\big]^2\,:\,\restrict{\vvh}{\P}\in\Vvhk(\P)\quad\forall\P\in\Th\Big\},
  \qquad\ks\geq1,
  \label{eq:VEM:global:space}
\end{align}
which is a conforming finite-dimensional subspace of
$\big[\HONEzr(\Omega)\big]^2$.
This functional space is obtained by ``gluing together'' the local
virtual element spaces $\Vvhk(\P)$ defined on the mesh elements
$\P\in\Th$.
We take $\Vvhk(\P)=\big[\Vshk(\P)\big]^2$ where $\Vshk(\P)$ is either
the scalar virtual element space originally introduced
in~\cite{BeiraodaVeiga-Brezzi-Cangiani-Manzini-Marini-Russo:2013}
\begin{align}
  \Vshk(\P):=\Big\{\vsh\in\HONE(\P)\,:\,
  \restrict{\vsh}{\partial\P}\in\CS{0}(\partial\P),\,
  \restrict{\vsh}{\E}\in\PS{k}(\E)\,\forall\E\in\partial\P,\,
  \Delta\vsh\in\PS{k-2}(\P)
  \Big\},
  \label{eq:BF:regular-space:def}
\end{align}
or its ``modified'' version (also called ``enhanced'' in the virtual
element literature) given by~\cite{Ahmad-Alsaedi-Brezzi-Marini-Russo:2013}
\begin{align}
  \Vshk(\P):=\Big\{\vsh\in\HONE(\P)\,:\,
  &\restrict{\vsh}{\partial\P}\in\CS{0}(\partial\P),\,
  \restrict{\vsh}{\E}\in\PS{k}(\E)\,\forall\E\in\partial\P,\,
  \Delta\vsh\in\PS{k}(\P),\nonumber\\[0.25em]
  &\int_{\P}\big(\vsh-\PinP{k}\vsh\big)\qs\dV=0\quad\forall\qs\in\PS{k}(\P)\backslash{\PS{k-2}(\P)}
  \Big\},
  \label{eq:BF:enhanced-space:def}
\end{align}
where $\PS{k}(\P)\backslash{\PS{k-2}(\P)}$ is the space of polynomials
of degree exactly equal to $k$ or $k-1$ and $\PinP{k}$ is the elliptic
projection operator that will be defined in
Section~\ref{subsec:degrees-of-freddom}, cf.
equations~\eqref{eq:Pinabla:def:A}-\eqref{eq:Pinabla:def:B}.
To avoid the proliferation of symbols practically denoting the same
thing, we use the notation ``$\Vshk$'' for the spaces
in~\eqref{eq:BF:regular-space:def}
and~\eqref{eq:BF:enhanced-space:def} with a minor abuse of notation.
The definition of the virtual element vector space in
\eqref{eq:BF:enhanced-space:def} has been modified according to the
so-called \emph{enhancement strategy}, so that the $\LTWO$-orthogonal
projection onto $\big[\PS{k}(\P)\big]^2$ of a virtual element function
$\vvh\in\Vvhk(\P)$ is computable.

\medskip
\begin{remark}[Connection with the Scott-Vogelius finite element method]
  The approximation spaces defined above contain the Scott-Vogelius
  finite element method as a special case.
  In fact, if we consider a triangular mesh, we recover the
  Scott-Vogelius FEM by setting $\Vvhk=\VvhkSV$, where
  \begin{align}
    \VvhkSV
    &:=\Big\{\vvh\in\big[\HONE(\Omega)\big]^2\,:\,\restrict{\vvh}{\P}\in\big[\PS{k}(\P)\big]^2\quad\forall\P\in\Th\Big\}, \label{eq:SV:vector:space:def}\\[0.5em]
    \Qshkkk
    &:=\Big\{\qsh\in\LTWOzr(\Omega)\,:\,\restrict{\qsh}{\P}\in\PS{\underline{k}}(\P)\quad\forall\P\in\Th\Big\}=\PS{\underline{k}}(\Th)\cap\LTWOzr(\Omega),\label{eq:SV:scalar:space:def}
  \end{align}
  for $k\geq1$ and $\underline{k}\leq\ks-1$, and taking
  $\ash(\cdot,\cdot)=\as(\cdot,\cdot)$,
  $\bsh(\cdot,\cdot)=\bs(\cdot,\cdot)$ in
  \eqref{eq:stokes:vem:A}-\eqref{eq:stokes:vem:B}.
  It is well-known that the scheme with $k=1$ and $\underline{k}=0$
  can be unstable on triangular and square meshes.
  It is worth noting that our virtual element method coincides with
  such approximations only on triangular and square meshes and only
  for $k=1$ and $\underline{k}=0$.
  In all the other cases, including for example $k=2$ and
  $\underline{k}=0$, our approach provides a different and
  well-behaving discretization.
  The performance of the VEM for different values of the integer pair
  $(k,\underline{k})$ on triangular and square meshes is investigated
  in the numerical section.
  Our numerical evidence shows that the VEM is stable and convergent
  except for the possibly unstable cases mentioned above.
\end{remark}

\begin{figure}[!t]
  \centering
  \begin{tabular}{ccc}
    \includegraphics[width=0.28\textwidth]{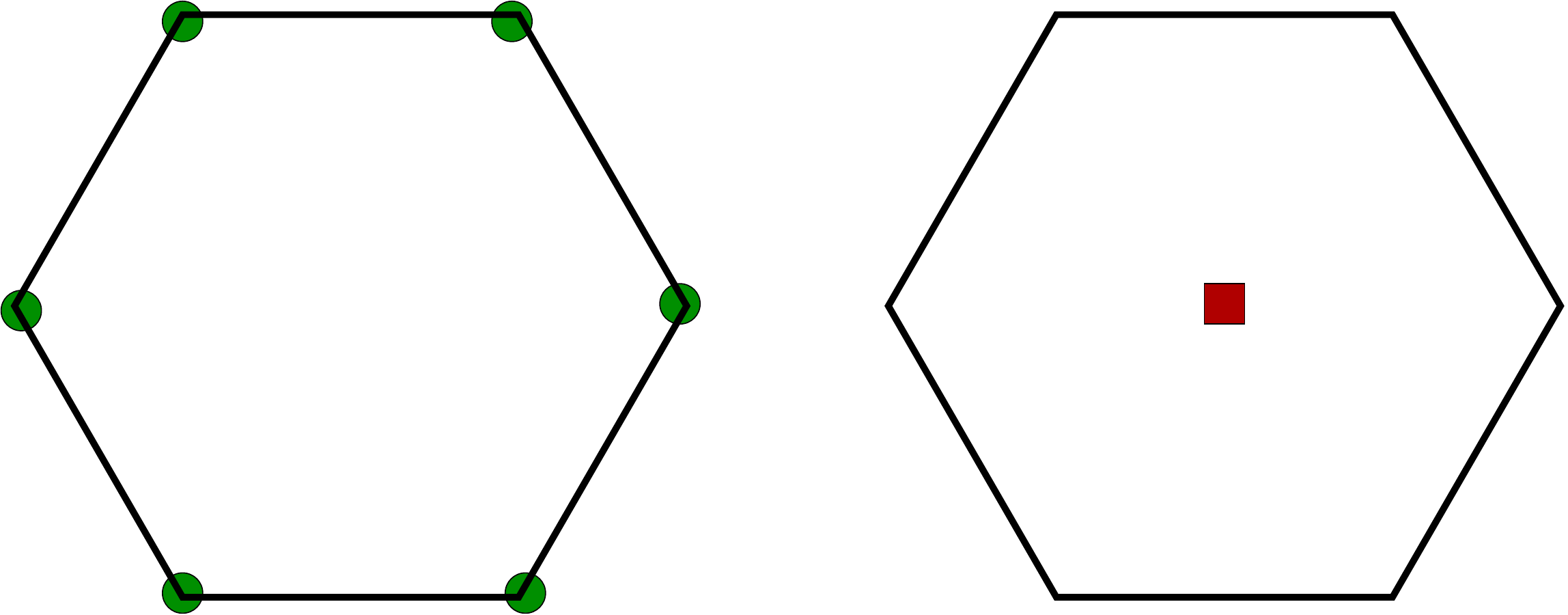} &\quad    
    \includegraphics[width=0.28\textwidth]{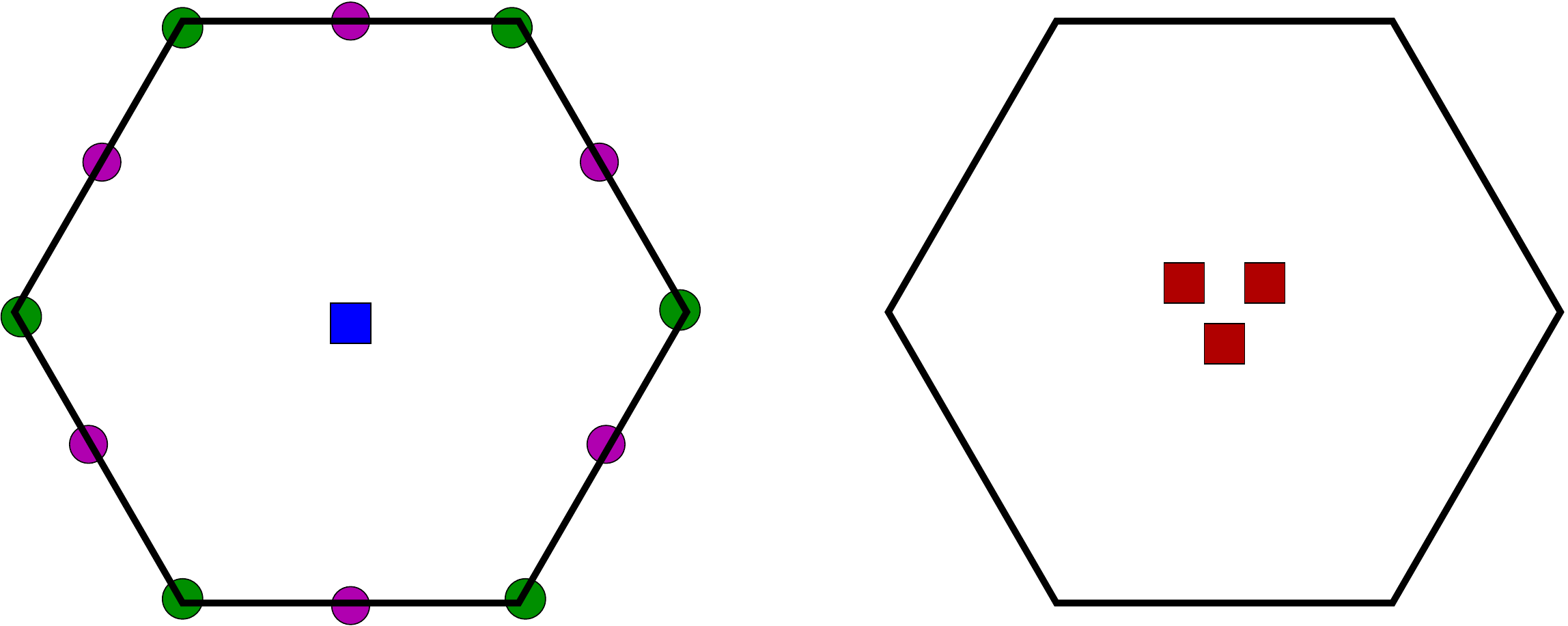} &\quad    
    \includegraphics[width=0.28\textwidth]{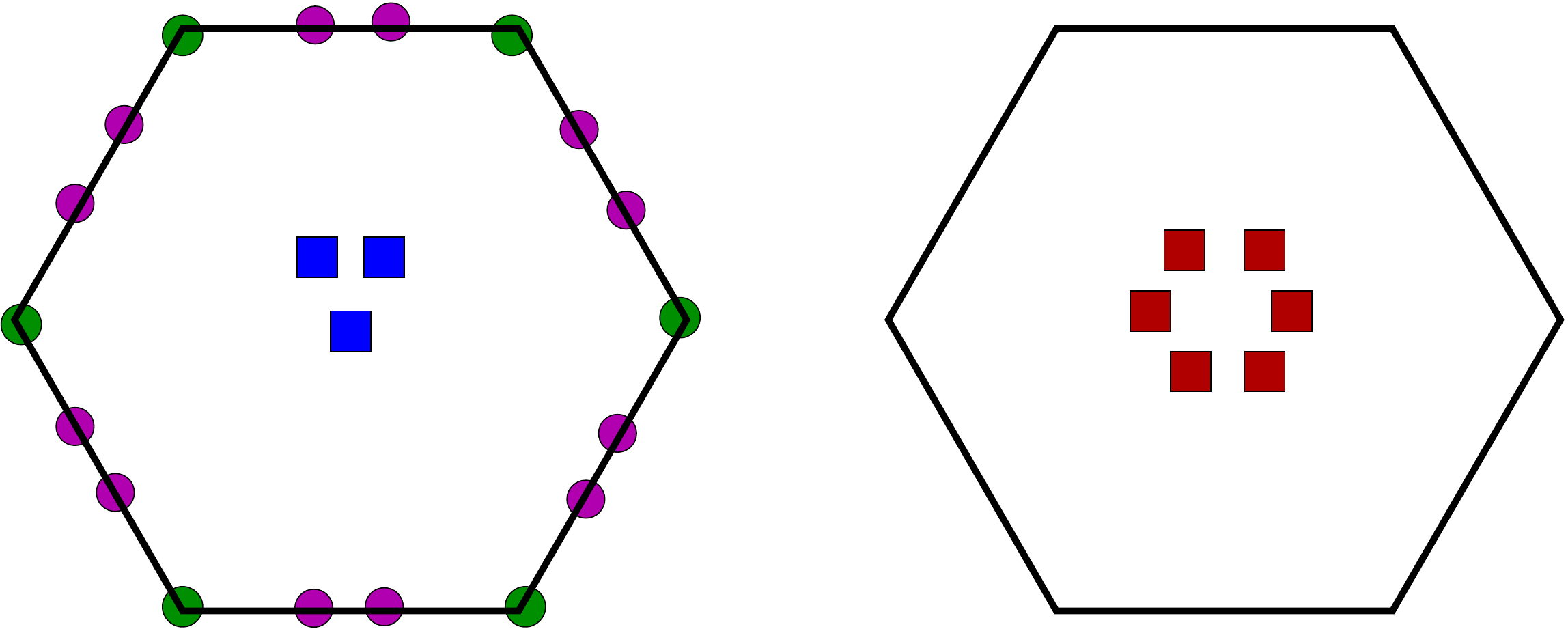} \\[0.5em] 
    \hspace{1mm}$\mathbf{(k,\underline{k})=(1,0)}$ &
    \hspace{4mm}$\mathbf{(k,\underline{k})=(2,1)}$ &
    \hspace{4mm}$\mathbf{(k,\underline{k})=(3,2)}$ 
  \end{tabular}
  \caption{Degrees of freedom of each component of the virtual element
    vector-valued fields of $\Vvhk(\P)$ (left) and the scalar
    polynomial fields of $\Qshkkk(\P)$ (right) of an hexagonal element
    for the accuracy degrees $k=1,2,3$ and $\underline{k}=k-1$.
    Vertex values and edge polynomial moments are marked by a circular
    bullet; cell polynomial moments are marked by a square bullet.}
  \medskip
  \label{fig:dofs:BF}
\end{figure}

\subsection{Degrees of freedom and polynomial projection operators}
\label{subsec:degrees-of-freddom}

The functions in the virtual element
space~\eqref{eq:BF:regular-space:def} and the enhanced space
\eqref{eq:BF:enhanced-space:def} are uniquely characterized by the
following set of values, the \emph{degrees of freedom}:

\medskip
\begin{itemize}
\item[\TERM{D1}] for $k\geq1$, the vertex values $\vsh(\xV)$,
  $\V\in\partial\P$;
\item[\TERM{D2}] for $k\geq2$, the polynomial edge moments
  \begin{align}
    \frac{1}{\mE}\int_{\E}\vsh(\ss)\qsh(\ss)\dS
    \qquad\forall\qsh\in\textrm{~basis~of~}\PS{k-2}(\E)
  \end{align}
  for every edge $\E\in\partial\P$;
\item[\TERM{D3}] for $k\geq2$, the polynomial cell moments
  \begin{align}
    \frac{1}{\mP}\int_{\P}\vsh(\xv)\qsh(\xv)\dV
    \qquad\forall\qsh\in\textrm{~basis~of~}\PS{k-2}(\P).
  \end{align} 
\end{itemize}

\medskip
The unisolvence of \TERM{D1}-\TERM{D3} is proved in
Reference~\cite{BeiraodaVeiga-Brezzi-Cangiani-Manzini-Marini-Russo:2013}
for the space defined in~\eqref{eq:BF:regular-space:def}, and in
Reference~\cite{Ahmad-Alsaedi-Brezzi-Marini-Russo:2013} for the
enhanced space defined in~\eqref{eq:BF:enhanced-space:def}.
In Figure~\ref{fig:dofs:BF}, we illustrate the degrees of freedom for
each velocity component and the scalar field for the polynomial degree
$k=1,2,3$ on an hexagonal element.
The degrees of freedom that uniquely characterize the virtual element
functions in the global space $\Vvhk$ are obtained by collecting the
elemental degrees of freedom.
Their unisolvence in the global space $\Vvhk$ is an immediate
consequence of the unisolvence of the degrees of freedom
\TERM{D1}-\TERM{D3} in the elemental spaces $\Vvhk(\P)$.

\begin{remark}
  In our implementation we used the scaled monomials and the
  orthogonal polynomials as two alternative bases in the definitions
  of the moments on $\E\in\partial\P$ and $\P$ of the degrees of
  freedom \TERM{D2} and \TERM{D3}.
  We remark that the choice of the basis of the polynomial space is
  arbitrary and does not change the theoretical convergence property
  of the method although it may impact on the conditioning of the
  discrete problem and practically affects the accuracy of the
  resulting approximation.
\end{remark}

\begin{remark}[Alternative choice of the degrees of freedom]
  Instead of \TERM{D2}, we can select the values at a set of $k-1$
  internal nodes on each edge, e.g., the nodes supporting the
  one-dimensional Gauss-Lobatto integration rule with $k-1$ internal
  nodes, which is exact on univariate polynomials of degree up to
  $2k+1$.
  This alternative choice is equivalent to \TERM{D2}.
  In the preliminary stages of our work, we carried out all numerical
  experiments with both choices of degrees of freedom and we did not
  find any significant difference in the results.
\end{remark}

Let $\vsh$ be a scalar virtual element function on the element $\P$
according to the definition given in~\eqref{eq:BF:regular-space:def}
or \eqref{eq:BF:enhanced-space:def}.
Then, the following polynomial projections are computable using the
degrees of freedom \TERM{D1}-\TERM{D3}:
\begin{itemize}
\item the elliptic projection $\PinP{k}\vsh\in\PS{k}(\P)$, which is
  the solution of the variational problem
  \begin{align}
    \int_{\P}\nabla\big(\vsh-\PinP{k}\vsh\big)\qs\dV &= 0 \qquad\forall\qs\in\PS{k}(\P),\label{eq:Pinabla:def:A}\\[0.5em]
    \int_{\partial\P}\big(\vsh-\PinP{k}\vsh\big)\dS    &= 0;\label{eq:Pinabla:def:B}
  \end{align}

\item the orthogonal projection $\PizP{\kb}\vsh\in\PS{\kb}(\P)$,
  which is the solution of the variational problem
  \begin{align*}
    \int_{\P}\big(\vsh-\PizP{\kb}\vsh\big)\qs\dV = 0
    \qquad\forall\qs\in\PS{\kb}(\P),
  \end{align*}
  where $\kb=k-2$ for $k\geq2$ if $\Vshk(\P)$ is defined by
  \eqref{eq:BF:regular-space:def}, and $\kb=k\geq1$ if $\Vshk(\P)$ is
  defined by \eqref{eq:BF:enhanced-space:def}.
\end{itemize}
Likewise, if $\vvh=(\vs_x,\vs_y)^T$ is a vector-valued field whose
components are in the virtual element spaces defined
by~\eqref{eq:BF:regular-space:def} or
\eqref{eq:BF:enhanced-space:def}, the following projectors are
computable:

\begin{itemize}
\item the orthogonal projection $\PizP{\kb}\vvh$ of a vector-valued
  field $\vvh=(\vs_x,\vs_y)^T$, which is the solution of the
  variational problem
  \begin{align*}
    \int_{\P}\big(\vvh-\PizP{\kb}\vvh\big)\cdot\qv\dV = 0
    \qquad\forall\qv\in\big[\PS{\kb}(\P)\big]^2,
  \end{align*}
  where $\kb=k-2$ for $k\geq2$ if each component of $\vvh$ belongs to
  the scalar virtual element space $\Vshk(\P)$ defined by
  \eqref{eq:BF:regular-space:def}, and $\kb=k\geq1$ if each component
  of $\vvh$ belongs to the enhanced virtual element space $\Vshk(\P)$
  defined by \eqref{eq:BF:enhanced-space:def}.
  The projection operator is computed component-wisely, i.e.,
  $\PizP{\kb}\vvh=(\PizP{\kb}\vs_x,\PizP{\kb}\vs_y)^T\in\big[\PS{\kb}(\P)\big]^2$,
  where $\PizP{\kb}\vs_x$ and $\PizP{\kb}\vs_y$ are the scalar
  orthogonal projections defined above;
  
\item the orthogonal projection
  $\PizP{k-1}\nabla\vvh\in\big[\PS{k-1}(\P)\big]^{2\times2}$ of
  $\nabla\vvh$, the gradient of the virtual element vector-valued
  field $\vvh$, onto the linear space of the $2\times2$-sized
  matrix-valued polynomials of degree $k-1$.
  These quantities are defined component-wisely as follows:
  \begin{align}
    \nabla\vvh &=
    \left(
    \begin{array}{cc}
      \frac{\partial\vs_x}{\partial\xs} & \quad\frac{\partial\vs_x}{\partial\ys}\\[1.0em]
      \frac{\partial\vs_y}{\partial\xs} & \quad\frac{\partial\vs_y}{\partial\ys}\\
    \end{array}
    \right)
    \intertext{and}\quad
    \PizP{k-1}\nabla\vvh &=
    \left(
    \begin{array}{cc}
      \PizP{k-1}\frac{\partial\vs_x}{\partial\xs} & \quad\PizP{k-1}\frac{\partial\vs_x}{\partial\ys}\\[1.0em]
      \PizP{k-1}\frac{\partial\vs_y}{\partial\xs} & \quad\PizP{k-1}\frac{\partial\vs_y}{\partial\ys}\\
    \end{array}
    \right),
  \end{align}
  and this latter one is the solution of the variational problem:
  \begin{align*}
    \int_{\P}\big(\nabla\vvh-\PizP{k-1}\nabla\vvh\big):\bm\kappa\dV = 0
    \qquad\forall{\bm\kappa}\in\big[\PS{k-1}(\P)\big]^{2\times2}.
  \end{align*}
\end{itemize}
The computability of the scalar polynomial projection is proved
in~\cite{BeiraodaVeiga-Brezzi-Cangiani-Manzini-Marini-Russo:2013,Ahmad-Alsaedi-Brezzi-Marini-Russo:2013},
while the computability of the vector polynomial projection can easily
be proved by using the same arguments of these papers
component-wisely.


\subsection{The virtual element bilinear forms $\ash(\cdot,\cdot)$}

First, we note that the linearity of the integral in \eqref{eq:as:def}
allows us to split the bilinear form $\as(\cdot,\cdot)$ as a summation
of the local bilinear forms defined on the mesh elements:
\begin{align}
  \as(\vv,\wv)&=\sum_{\P\in\Th}\asP(\vv,\wv)\quad\textrm{with}\quad
  \asP(\vv,\wv)=\int_{\P}\nabla\vv:\nabla\wv\dV.
  \label{eq:asP:def}
\end{align}
The virtual element bilinear form $\ash(\cdot,\cdot)$ can be defined
as the sum of the local bilinear forms
$\asPh(\cdot,\cdot):\Vvhk(\P)\times\Vvhk(\P)\to\REAL$ by using the
orthogonal projection operators:
\begin{align}
  &\ash(\vvh,\wvh) = \sum_{\P\in\Th}\asPh(\vvh,\wvh),
  \label{eq:ash:def}
  \intertext{where~}
  &\asPh(\vvh,\wvh)
  = \int_{\P}\PizP{k-1}\nabla\vvh:\PizP{k-1}\nabla\wvh\dV
  + \SPh\big( (1-\Pi^{\P}_{k})\vvh, (1-\Pi^{\P}_{k})\wvh \big),
  \label{eq:asPh:def}
\end{align}
where $\Pi^{\P}_{k}$ in the last term can be either the orthogonal or
the elliptic projection $\PizP{k}$ or $\PinP{k}$.
The definition in~\eqref{eq:asPh:def} includes the stabilization term
provided by the local bilinear form
$\SPh(\cdot,\cdot):\Vvhk(\P)\times\Vvhk(\P)\to\REAL$.
According to the virtual element setting, for $\SPh(\cdot,\cdot)$ we
can choose any symmetric, positive definite bilinear form such that
\begin{align*}
  \sigma_*\as(\vvh,\vvh)\leq\SPh(\vvh,\vvh)\leq\sigma^*\as(\vvh,\vvh)
  \qquad
  \forall\vvh\in\Vvhk(\P)\cap\textrm{ker}(\Pi^{\P}_{k}),
\end{align*}
where $\sigma_*$ and $\sigma^*$ are two real, positive constants
independent of $\hh$.
A detailed study of possible stabilization forms can be found
in~\cite{Mascotto:2018}.
This design is such that the following two fundamental properties are
true for the local bilinear form $\asPh(\cdot,\cdot)$:
\begin{itemize}
\item \textbf{Stability}: there exist two real, positive constants
  $\alpha_*$ and $\alpha^*$ such that
  \begin{align}
    \alpha_*\asP(\vvh,\vvh)\leq\asPh(\vvh,\vvh)\leq\alpha^*\asP(\vvh,\vvh)
    \qquad\forall\vvh\in\Vvhk(\P).
    \label{eq:ash:stability}
  \end{align}
  These constants are independent of $\hh$ but can depend on other
  parameters of the discretization, such as $\rho$, which is related
  to the regularity of the mesh, $\sigma_*$ and $\sigma^*$, and the
  order of the method $k$.

  \medskip
\item \textbf{Polynomial consistency}: the \emph{exactness property}
  holds
  \begin{align}
    \asPh(\vvh,\qv) = \asP(\vv,\qv),
    \label{eq:consistency}
  \end{align}
  for every vector field $\vvh\in\Vvhk(\P)$ and vector polynomial
  field $\qv\in\big[\PS{k}(\P)\big]^2$.
  \medskip
\end{itemize}
These properties have two important consequences.
First, adding all the local left inequalities
in~\eqref{eq:ash:stability} implies that the bilinear form
$\ash(\cdot,\cdot)$ is coercive on $\Vvhk\times\Vvhk$:
\begin{align}
  \ash(\vvh,\vvh) \geq \alpha_*\snorm{\vvh}{1,\Omega}^2.
  \label{eq:ash:coercivity}
\end{align}
Second, the local bilinear form $\asPh(\cdot,\cdot)$ is continuous on
$\Vvhk\times\Vvhk$.
Indeed, relation~\eqref{eq:ash:stability} and the symmetry of
$\asPh(\cdot,\cdot)$ imply that $\asPh(\cdot,\cdot)$ is a semi-inner
product on $\Vvhk(\P)$ (indeed, it is an inner product on the
quotient space $\big[\Vshk(\P)\backslash{\REAL}\big]^2$).
Using the Cauchy-Schwarz inequality we find that:
\begin{align}
  \asPh(\vvh,\wvh)
  &\leq \big[\asPh(\vvh,\vvh)\big]^{\frac{1}{2}}\,\big[\asPh(\wvh,\wvh)\big]^{\frac{1}{2}}
  \leq \alpha^*\,\big[\asP(\vvh,\vvh)\big]^{\frac{1}{2}}\,\big[\asP(\wvh,\wvh)\big]^{\frac{1}{2}}
  \nonumber\\[0.5em]
  &=    \alpha^*\,\snorm{\vvh}{1,\P}\,\snorm{\wvh}{1,\P}.
  \label{eq:ash:local:continuity}
\end{align}
Then, we sum all the local inequalities and use again the
Cauchy-Schwarz inequality and the right inequality
of~\eqref{eq:ash:stability} to find that
\begin{align}
  \ash(\vvh,\wvh)
  &=\sum_{\P\in\Th}\asPh(\vvh,\wvh)
  \leq \alpha^*\,\sum_{\P\in\Th}\,\snorm{\vvh}{1,\P}\,\snorm{\wvh}{1,\P}
  \nonumber\\[0.5em]
  &\leq \alpha^*\,\Bigg(\sum_{\P\in\Th}\,\snorm{\vvh}{1,\P}^2\Bigg)^{\frac12}\,\Bigg(\sum_{\P\in\Th}\,\snorm{\wvh}{1,\P}^2\Bigg)^{\frac12}
  =    \alpha^*\,\snorm{\vvh}{1,\Omega}\,\snorm{\wvh}{1,\Omega},
  \label{eq:ash:global:continuity}
\end{align}
which implies the global continuity of $\ash$.

\subsection{The virtual element bilinear forms $\bsh(\cdot,\cdot)$}

First, we note that the linearity of the integral in \eqref{eq:bs:def}
allows us to split the bilinear form $\bs(\cdot,\cdot)$ as a summation
of the local bilinear forms defined on the mesh elements:
\begin{align}
  \bs(\vv,\qs)&=\sum_{\P\in\Th}\bsP(\vv,\wv)\quad\textrm{with}\quad
  \bsP(\vv,\qs)=-\int_{\P}\qs\,\DIV\vv\dV.
  \label{eq:bsP:def}
\end{align}

We split $\bsh(\cdot,\cdot)$ as the sum of the local bilinear forms
$\bsPh(\cdot,\cdot):\Vvhk(\P)\times\PS{\underline{k}}(\P)\to\REAL$ to
define its virtual element approximation:
\begin{align}
  \bsh(\vvh,\qs) = \sum_{\P\in\Th}\bsPh(\vvh,\qs)
  \qquad\textrm{where}\qquad
  \bsPh(\vvh,\qs)
  = -\int_{\P}\qs\,\PizP{\underline{k}}\DIV\vvh\dV.
  \label{eq:bsh:def}
\end{align}
It holds that $\bsPh(\vvh,\qs)=\bsP(\vvh,\qs)$ for every vector-valued
field $\vvh\in\Vvhk(\P)$ and polynomial scalar function
$\qs\in\PS{\underline{k}}(\P)$, where we recall that
$\bsP(\cdot,\cdot)$ is defined in~\eqref{eq:bsP:def}.
This property is a straightforward consequence of the definition of
the orthogonal projection operator $\PizP{\underline{k}}$.
If we add this relation over all the mesh elements, we find that
\begin{align}
  \bsh(\vvh,\qs) = \bs(\vvh,\qs)
  \qquad\forall\vvh\in\Vvhk,\,\qs\in\PS{\underline{k}}(\Th)
  \label{eq:bsh-bs}
\end{align}
for all pairs $(k,\underline{k})$ with $\underline{k}\leq\ks-1$, which
can be used to simplify the implementation of the method.
  
\subsection{The virtual element approximation of the right-hand side}
Let $\kb$ be a nonnegative integer number.
The right hand-side of equation~\eqref{eq:stokes:vem:A} is written as
\begin{align}
  \bil{\fvh}{\vvh}
  = \sum_{\P\in\Th}\int_{\P}\PizP{\kb}\fv\cdot\vvh\dV 
  = \sum_{\P\in\Th}\int_{\P}\fv\cdot\PizP{\kb}\vvh\dV,
\end{align}
where the vector-valued source term $\fv$ is locally approximated by
its orthogonal projection $\PizP{\kb}\fv$ onto the polynomial space
$\PS{\kb}(\P)$ for all polygonal elements $\P$.
Note that we use the definition of the orthogonal projector
$\PizP{\kb}$ in the second equality above.

Two choices of $\kb$ are possible:
\begin{itemize}
\item[$\bullet$] $\kb=\max(0,k-2)$ if we consider the virtual element
  space in~\eqref{eq:BF:regular-space:def};
\item[$\bullet$] $\kb=k$ (or $k-1$): if we consider the enhanced
  virtual element space in~\eqref{eq:BF:enhanced-space:def}.
\end{itemize}
The first choice was proposed for $k\geq2$ in the original
paper~\cite{BeiraodaVeiga-Brezzi-Cangiani-Manzini-Marini-Russo:2013}.
It is worth mentioning that for $k=1$ the projector $\PizP{0}$ is
substituted by the arithmetic average of the elemental vertex values
for each velocity vector component.
This choice allows us to obtain the correct convergence rate for the
approximation of the velocity field in the energy norm.
The second choice was proposed in
Ref.~\cite{Ahmad-Alsaedi-Brezzi-Marini-Russo:2013} and allows us to
obtain the correct convergence rate for the approximation of the
velocity field in the $\LTWO$-norm.


We recall the error estimates pertaining these two possible
approximations of the right-hand side, which follow on noting that
$\big(1-\PizP{\kb}\big)\fvh$ is orthogonal to $\PizP{0}\vvh$ in the
$\LTWO$-inner product.
Assuming that $\fv\in\big[\HS{\kb+1}(\Omega)\big]^{\dims}$
with $\kb\geq0$ we find that~\cite{BeiraodaVeiga-Brezzi-Cangiani-Manzini-Marini-Russo:2013}:
\begin{align}
  \ABS{\bil{\fvh}{\vvh}-\scal{\fv}{\vvh}}
  &= \ABS{ \sum_{\P\in\Th}\int_{\P}\big(\PizP{\kb}\fv-\fv\big)\vvh\dV }
  \nonumber\\[0.5em]
  &\leq \sum_{\P\in\Th}\ABS{ \int_{\P}\big(\PizP{\kb}\fv-\fv\big)\big(\vvh-\PizP{0}\vvh\big)\dV }
  \nonumber\\[0.5em]
  &\leq \sum_{\P\in\Th} \norm{\PizP{\kb}\fv-\fv}{0,\P}\,\norm{\vvh-\PizP{0}\vvh}{0,\P}
  \nonumber\\[0.5em]
  &\leq \Cs\hh^{\kb+2}\norm{\fv}{\kb+1,\P}\,\snorm{\vvh}{1,\P},
  \label{eq:fv:bound:0}
\end{align}
for some constant $C>0$ independent of $\hh$.
For $\kb=0$ and assuming $\fv\in\big[\LTWO(\Omega)\big]^{\dims}$, we
find that
\begin{align}
  \ABS{\bil{\fvh}{\vvh}-\scal{\fv}{\vvh}}
  &= \ABS{ \sum_{\P\in\Th}\int_{\P}\big(\PizP{0}\fv-\fv\big)\vvh\dV }
  \nonumber\\[0.5em]
  &\leq \sum_{\P\in\Th}\ABS{ \int_{\P}\big(\PizP{0}\fv-\fv\big)\big(\vvh-\PizP{0}\vvh\big)\dV }
  \nonumber\\[0.5em]
  &\leq \sum_{\P\in\Th} \norm{\PizP{0}\fv-\fv}{0,\P}\,\norm{\vvh-\PizP{0}\vvh}{0,\P}
  \nonumber\\[0.5em]
  &\leq \Cs\hh\norm{\fv}{0,\P}\,\snorm{\vvh}{1,\P},
  \label{eq:fv:bound:1}
\end{align}
for some constant $C>0$ independent of $\hh$.

\section{Convergence theory in the energy norm}
\label{sec4:convergence}

In this section, we present an abstract convergence result for the VEM
based on the pair of finite dimensional functional spaces
$\big(\Vvhk,\Qshkk\big)$, see Theorem~\ref{theorem:H1:abstract}.
This result allows us to derive an estimate for the approximation
error on the velocity and the pressure, see
Corollary~\ref{corollary:error:estimate}.
To carry out the convergence analysis in
Section~\ref{subsec:convergence:energy:norm}, we need a few
preliminary results, which are reported in
Section~\ref{subsec:preliminary:results}, and to prove that the
bilinear form $\bsh(\cdot,\cdot)$ is inf-sup stable on
$\Vvhk\times\Qshkk$ on the mesh families satisfiyng
Assumptions~\textbf{(M1)}-\textbf{(M2)}, see
Section~\ref{subsec:inf-sup}.

\subsection{Preliminary results}
\label{subsec:preliminary:results}

In the analysis of the next sections, we make use of the interpolants
of the (smooth enough) vector-valued field $\vv$ and scalar function
$\qs$, respectively.
The interpolants are the vector-valued field $\vvI$ in $\Vvhk$ and the
scalar polynomial $\qsI$ in $\Qshkk$ that have the same degrees of
freedom of $\vv$ and $\qs$.
In this section, we report three technical lemmas regarding $\vvI$ and
$\qsI$.
The first two lemmas state basic results for the approximation of a
vector-valued field by its virtual element interpolant and a scalar
function by its polynomial interpolant onto the subspace of
polynomials, and are presented without a proof.
The third lemma presents an identity that is used in the proof of
Theorem~\ref{theorem:H1:abstract}.

\medskip
\begin{lemma}[Projection error~\cite{BeiraodaVeiga-Brezzi-Cangiani-Manzini-Marini-Russo:2013,Brenner-Scott:1994}]
  \label{lemma:projection:error}
  Under Assumptions~\textbf{(M1)}-\textbf{(M2)}, for every
  vector-valued field $\vv\in\big[\HS{s+1}(\P)\big]^2$ and scalar
  function $\qs\in\HS{s}(\P)$ with $1\leq\ss\leq\ell$, there exists a
  vector polynomial $\vv_{\pi}\in\big[\PS{\ell}(\P)\big]^2$ and a
  scalar polynomial $\qs_{\pi}\in\PS{\ell-1}(\P)$ such that
  \begin{align}
    &\norm{\vv-\vv_{\pi}}{0,\P} + \hP\snorm{\vv-\vv_{\pi}}{1,\P}\leq\Cs\hP^{s+1}\snorm{\vv}{s+1,\P},\\[0.75em]
    &\norm{\qs-\qs_{\pi}}{0,\P} + \hP\snorm{\qs-\qs_{\pi}}{1,\P}\leq\Cs\hP^{s} \snorm{\qs}{s,\P}
  \end{align}
  for some positive constant $\Cs$ that is independent of $\hP$ but
  may depend on the polynomial degree $\ell$ and the mesh regularity
  constant $\varrho$.
\end{lemma}

\medskip
\begin{lemma}[Interpolation error~\cite{BeiraodaVeiga-Brezzi-Cangiani-Manzini-Marini-Russo:2013,Brenner-Scott:1994}]
  \label{lemma:interpolation:error}
  Under Assumptions~\textbf{(M1)}-\textbf{(M2)}, for every
  vector-valued field $\vv\in\big[\HS{s+1}(\P)\big]^2$ and scalar
  function $\qs\in\HS{s}(\P)$ with $1\leq\ss\leq\ell$, there exists a
  vector valued-field $\vvI\in\Vvhk(\P)$ and a scalar field
  $\qsI\in\PS{\ell-1}(\P)$ such that
  \begin{align}
    \norm{\vv-\vvI}{0,\P} + \hP\snorm{\vv-\vvI}{1,\P}\leq\Cs\hP^{s+1}\snorm{\vv}{s+1,\P},\\[0.5em]
    \norm{\qs-\qsI}{0,\P} + \hP\snorm{\qs-\qsI}{1,\P}\leq\Cs\hP^{s} \snorm{\qs}{s,\P}
  \end{align}
  for some positive constant $\Cs$ that is independent of $\hP$ but
  may depend on the polynomial degree $\ell$ and the mesh regularity
  constant $\varrho$.
\end{lemma}

\medskip
\begin{lemma}[Orthogonality between $\DIV(\uvh-\uvI)$ and $\psh-\psI$]
  Let $\uv\in\big[\HONEzr(\Omega)\big]^2$ be the exact solution of the
  variational formulation of the Stokes problem given
  in~\eqref{eq:stokes:var:A}-\eqref{eq:stokes:var:B}.
  Let $(\uvh,\psh)\in\Vvhk\times\Qshkk$ be the solution of the virtual
  element
  approximation~\eqref{eq:stokes:vem:A}-\eqref{eq:stokes:vem:B}.
  Then, it holds that
  \begin{align}
    \bs(\uvh-\uvI,\psh-\psI) &= 0.
    \label{eq:aux:20}
  \end{align}
\end{lemma}
\begin{proof}
  First, note that the composed operator $\PizP{k-1}\DIV(\cdot)$ only
  depends on the degrees of freedom of its argument.
  These degrees of freedom are the same for $\uv$ and
  its virtual element interpolation $\uvI$, so that it must hold
  that $\PizP{k-1}\DIV\uv=\PizP{k-1}\DIV\uvI$.
  Using this property and the definition of the orthogonal projection
  $\PizP{k-1}$ yield
  \begin{align*}
    \bs(\uv,\qs)
    &= \sum_{\P}\int_{\P}\qs\,\DIV\uv\dV
    = \sum_{\P}\int_{\P}\qs\PizP{k-1}\DIV\uv\dV
    = \sum_{\P}\int_{\P}\qs\PizP{k-1}\DIV\uvI\dV
    \\[0.5em]
    &= \sum_{\P}\int_{\P}\qs\,\DIV\uvI\dV
    = \bs(\uvI,\qs)
  \end{align*}
  for every $\qs\in\PS{k-1}(\P)$.
  Eventually, we note that $\bs(\uvI,\qs)=\bs(\uv,\qs)=0$ from
  Eq.~\eqref{eq:stokes:var:B} and $\bs(\uvh,\qs)=\bsh(\uvh,\qs)=0$
  from Eqs.~\eqref{eq:bsh-bs} and~\eqref{eq:stokes:vem:B}.
  So, it holds that $\bs(\uvh-\uvI,\qs)=0$, and relation~\eqref{eq:aux:20}
  immediately follows by setting $\qs=\psh-\psI$.
\end{proof}

\subsection{Discrete Inf-sup stability}
\label{subsec:inf-sup}

The main result of this section is the \emph{discrete inf-sup
stability} property that is stated by the following theorem.
\begin{theorem}
  \label{theorem:discrete:infsup}
  Let $\Vvhk$ and $\Qshkk$ the virtual element space and the piecewise
  polynomial space respectively defined on meshes satisfying
  Assumptions $\textbf{(M1)}-\textbf{(M2)}$ for $k\geq2$
  in~\eqref{eq:VEM:global:space} and~\eqref{eq:SV:scalar:space:def},
  where $\Vvhk(\P)$ can be either the regular local VEM space defined
  by~\eqref{eq:BF:regular-space:def} or the enhanced VEM space defined
  by~\eqref{eq:BF:enhanced-space:def}.
  Then, there exists a real positive constant $\beta$ (independent of
  $\hh$) such that
  \begin{align}
    \inf_{\qsh\in\Qshkk\backslash\{0\}}\sup_{\vvh\in\Vvhk\backslash\{\zerov\}}\frac{\bsh(\vvh,\qsh)}{ \Norm{\qsh}{0,\Omega}\,\snorm{\vvh}{1,\Omega} }
    \geq\beta.
    \label{eq:discrete:infsup}
  \end{align}
\end{theorem}
\begin{proof}
  To prove this result, we exploit the inf-sup inequality
  of~\cite[Proposition~4.3]{BeiraodaVeiga-Lovadina-Vacca:2017} and the
  result
  of~\cite[Lemma~2 and~Proposition~1]{Chernov-Marcati-Mascotto:2021}.
  We report these results below for completeness of exposition, with
  some (very minor) notational modifications to adapt them to the
  present context.
  The proof and a detailed discussion can be found in the cited
  references.
  \begin{proposition}[Proposition~4.3 from Reference~\cite{BeiraodaVeiga-Lovadina-Vacca:2017}]
    \label{prop:4.3}
    Consider the virtual element space of vector-valued functions:
    \begin{align}
      \Vvhkt =
      \Big\{
      \vvht\in\big[\HONEzr(\Omega)\big]^2\,\textrm{such~that}~\restrict{\vvht}{\P}\in\Vvhkt(\P)
      \,\,\forall\P\in\Th
      \Big\},
    \end{align}
    where the local vector space on every element $\P$ is defined as
    \begin{align}
      \Vvhkt(\P) =
      \Big\{\,
      \vvht\in\big[\HONE(\P)\big]^2\,:\,
      &\restrict{\vvht}{\partial\P}\in\big[\CS{0}(\partial\P)\big]^2,\,
      \restrict{\vvht}{\E}\in\big[\PS{k}(\E)\big]^2\,\forall\E\in\partial\P,\,\nonumber\\[0.25em]
      &-\Delta\vvht-\nabla\ss\in\calG_{k-2}(\P)^{\perp}\quad\textrm{for~some~}\ss\in\LTWO(\P),\nonumber\\[0.25em]
      &\DIV\vvht\in\PS{k-1}(\P)\,
      \Big\},
    \end{align}
    and $\calG_{k-2}(\P)^{\perp}$ is the $\LTWO$-orthogonal complement
    to $\calG_{k-2}(\P)=\nabla\PS{k-1}(\P)$.
    Then, there exists a real positive constant $\tbeta$ (independent
    of $\hh$) such that
    \begin{align}
      \inf_{\qsh\in\Qshkk\backslash\{0\}}\sup_{\vvht\in\Vvhkt\backslash\{\zerov\}}\frac{\bsh(\vvht,\qsh)}{ \Norm{\qsh}{0,\Omega}\,\snorm{\vvht}{1,\Omega} }
      \geq\tbeta.
      \label{eq:discrete:infsup:Prop4.3}
    \end{align}
  \end{proposition}
  The virtual element space $ \Vvhkt$ that is considered in
  Proposition~\ref{prop:4.3} is rather different than the one that is
  considered in our work.
  However, it is proved
  in~\cite[Lemma~2]{Chernov-Marcati-Mascotto:2021} that we can
  establish a bijective correspondance mapping the degrees of freedom
  characterizing the ``Stokes-like'' virtual element fields in
  $\Vvhkt$ (see also~\cite{BeiraodaVeiga-Lovadina-Vacca:2017} for
  their definition) and those characterizing the ``Poisson-like''
  virtual element fields that are discussed in
  Section~\ref{subsec:degrees-of-freddom},
  cf.~\eqref{eq:VEM:global:space}.
  Since the degrees of freedom are the same for both the regular and
  the enhanced space definitions~\eqref{eq:BF:regular-space:def}
  and~\eqref{eq:BF:enhanced-space:def}, the argument that we present
  below holds for both situations.
  Let $\T^{\P}_{S|P}:\Vvhkt(\P)\to\Vvhk(\P)$ denote this mapping.
  We report the result
  of~\cite[Proposition~1]{Chernov-Marcati-Mascotto:2021} as follows.
  \begin{proposition}[Proposition~1 from \cite{Chernov-Marcati-Mascotto:2021}]
    \label{prop:1}
    For all $\vvht\in\Vvhkt$ we have
    \begin{align}
      &\bs(\T^{\P}_{S|P}\vvht,\qs) = \bs(\vvht,\qs)\quad\forall\qs\in\Qshk
      \intertext{and}
      &\Pin{k}(\T^{\P}_{S|P}\vvht)=\Pin{k}(\vvht),\quad
      \Piz{k-2}(\T^{\P}_{S|P}\vvht)=\Piz{k-2}(\vvht).
    \end{align}
  \end{proposition}
  The virtual element spaces $\Vvhk$ (this paper) and $\Vvhkt$
  (Reference~\cite{BeiraodaVeiga-Lovadina-Vacca:2017}) are
  substantially different although they both have the same vector
  polynomial subspace.
  Moreover, the virtual element vector-valued fields in them are
  uniquely identified by rather different sets of degrees of freedom
  (we refer again to~\cite{BeiraodaVeiga-Lovadina-Vacca:2017}
  for their definition in $\Vvhkt$).
  Nonetheless, Proposition~\ref{prop:1} above implies that two
  functions $\vvh\in\Vvhk$ and $\vvht\in\Vvhkt$ have the same
  polynomial projections whenever they are in the bijection
  correspondance through $\T^{\P}_{S|P}$.
  As a consequence, the discrete bilinear forms $\ash(\cdot,\cdot)$
  and $\bsh(\cdot,\cdot)$ and the right-hand side functional
  $\bil{\fvh}{\cdot}$, which only depend on such projection operators,
  must evaluate the same on corresponding functions.
  So, starting from Proposition~\ref{prop:1}, i.e.,
  Reference~\cite[Proposition~1]{Chernov-Marcati-Mascotto:2021}, we
  find that
  \begin{align*}
    \begin{array}{rll}
      \bs(\vvh ,\qsh)
      &=       \bs(\vvht,\qsh)                                                                                  & \mbox{\big[use~\cite[Proposition~4.3]{BeiraodaVeiga-Lovadina-Vacca:2017} \big]} \\[0.5em]
      &\geq    \tbeta                                      \snorm{\vvht}{1,\Omega}  \,\norm{\qsh}{0,\Omega} & \mbox{\big[use~stability~inequality~\cite[Eq.(3.25)]{BeiraodaVeiga-Lovadina-Vacca:2017}\big]} \\[0.5em]
      &\geq    \tbeta\big(\TILDE{\alpha}^*\big)^{-1}        \ash(\vvht,\vvht)^{\frac12}\,\norm{\qsh}{0,\Omega} & \mbox{\big[use~again~\cite[Proposition~1]{Chernov-Marcati-Mascotto:2021}\big]} \\[0.5em]
      &=       \tbeta\big(\TILDE{\alpha}^*\big)^{-1}        \ash(\vvh,\vvh)^{\frac12}  \,\norm{\qsh}{0,\Omega} & \mbox{\big[use~\eqref{eq:ash:stability}\big]} \\[0.5em]
      &\geq    \tbeta\big(\TILDE{\alpha}^*\big)^{-1}\alpha_*\snorm{\vvh}{1,\Omega}   \,\norm{\qsh}{0,\Omega},
    \end{array}
  \end{align*}
  which is the assertion of the theorem after setting
  $\beta=\tbeta\big(\TILDE{\alpha}^*\big)^{-1}\alpha_*$ and
  noting that this real constant is independent of $\hh$.
\end{proof}








\begin{remark}
  An alternative proof of the inf-sup inequality is also possible by
  adapting the argument used to prove such inequality in the virtual
  element method for linear elasticity that is found
  in~\cite{BeiraodaVeiga-Brezzi-Marini:2013}.
\end{remark}

\subsection{Convergence in the energy norm}
\label{subsec:convergence:energy:norm}

\begin{theorem}[Abstract convergence result]
  \label{theorem:H1:abstract}
  Let
  $(\uv,\ps)\in\big[\HS{1+s}(\Omega)\cap\HONEzr(\Omega)\big]^2\times\LTWOzr(\Omega)$,
  $s\geq1$, denote the solution of the variational formulation of the
  Stokes problem given
  in~\eqref{eq:stokes:var:A}-\eqref{eq:stokes:var:B}.
  Let $(\uvh,\psh)\in\Vvhk\times\Qshkk$ denote the solution of the
  virtual element variational formulation
  \eqref{eq:stokes:vem:A}-\eqref{eq:stokes:vem:B} for any polynomial
  degree $k\geq2$ on the mesh family $\{\Th\}_{\hh}$ satisfying the
  mesh regularity assumptions $\textbf{(M1)}-\textbf{(M2)}$.
  Then, the solution pair $(\uvh,\psh)\in\Vvhk\times\Qshkk$ exists and
  is unique, and for every interpolation approximation $\uvI\in\Vvhk$
  and piecewise polynomial approximation
  $\uv_{\pi}\in\big[\PS{k}(\Th)\big]^2$, the following abstract
  estimate holds:
  \begin{align}
    \snorm{\uv-\uvh}{1,\Omega} + \norm{\ps-\psh}{0,\Omega}
    \leq\Cs\Bigg(
    &
    \snorm{\uv-\uvI}{1,\Omega}
    + \snorm{\uv-\uv_{\pi}}{1,\hh}
    + \norm{\ps-\psI}{0,\Omega}
    \nonumber\\[0.125em] &
    + \sup_{\vvh\in\Vvhk\backslash\{\zerov\}}\frac{\ABS{\bil{\fvh}{\vvh}-\scal{\fv}{\vvh}}}{\snorm{\vvh}{1,\Omega}}
    \Bigg)
    \label{eq:theo:abstract}
  \end{align}
  for some real, strictly positive constant $\Cs$ independent of
  $\hh$.
\end{theorem}

\begin{proof}
  First, we note that the stability property~\eqref{eq:ash:stability}
  and the Cauchy-Schwarz inequality imply that
  \begin{align}
    \Cs\snorm{\vvh}{1,\Omega}^2
    &\leq
    \ash(\vvh,\vvh)
    \phantom{\Cs\snorm{\vvh}{1,\Omega}\,\snorm{\wvh}{1,\Omega}}\hspace{-0.5cm}
    \forall\vvh\in\Vvhk,
    \label{eq:coercivity}
    \\[0.5em]
    \ABS{\ash(\vvh,\wvh)}
    &\leq
    \Cs\snorm{\vvh}{1,\Omega}\,\snorm{\wvh}{1,\Omega}
    \phantom{\ash(\vvh,\vvh)}\hspace{-0.5cm}
    \forall\vvh,\,\wvh\in\Vvhk,
    \label{eq:continuity}
  \end{align}
  for some real, strictly positive constant $\Cs$.
  The bilinear form $\bsh(\cdot,\cdot)$ is bounded since an
  application of the Cauchy-Schwarz inequality yields that
  \begin{align*}
    \ABS{\bsh(\vvh,\qs)}\leq\Cs\snorm{\vvh}{1,\Omega}\,\norm{\qs}{0,\Omega}
    \qquad\forall\vvh\in\Vvhk,\,\qs\in\Qshkk.
  \end{align*}
  According to the theory of saddle-point problems,
  see~\cite{Boffi-Brezzi-Fortin:2013}, the well-posedness of the VEM
  follows from the coercivity and boundedness of the bilinear form
  $\ash(\cdot,\cdot)$ and the discrete inf-sup
  inequality~\eqref{eq:discrete:infsup} and boundedness of the
  bilinear form $\bsh(\cdot,\cdot)$.

  \medskip
  To prove inequality~\eqref{eq:theo:abstract}, we introduce the
  virtual element field $\dvh=\uvh-\uvI\in\Vvhk$.
  The coercivity of $\ash(\cdot,\cdot)$,
  cf. Eq.~\eqref{eq:coercivity}, implies that:
  \begin{align*}
    \Cs\snorm{\dvh}{1,\Omega}^2
    \leq \ash(\dvh,\dvh)
    = \ash(\uvh,\dvh) - \ash(\uvI,\dvh).
  \end{align*}
  Then, we use~\eqref{eq:stokes:vem:A} and~\eqref{eq:ash:def} and add
  and substract $\uv_{\pi}$, we use~\eqref{eq:bsh-bs} and
  ~\eqref{eq:aux:20}, which implies that
  $\bsh(\dvh,\psh)=\bs(\dvh,\psh)=\bs(\dvh,\psI)$, we use the
  consistency relation~\eqref{eq:consistency} for every local bilinear
  form $\asPh(\cdot,\cdot)$,
  and we finally add and substract $\uv$ to the last term to obtain:
  \begin{align*}
    \Cs\snorm{\dvh}{1,\Omega}^2
    &\leq\bil{\fvh}{\dvh} - \bsh(\dvh,\psh) - \sum_{\P\in\Th}\Big(\asPh(\uvI-\uv_{\pi},\dvh) + \asPh(\uv_{\pi},\dvh) \Big)
    \\[0.5em]
    &= \bil{\fvh}{\dvh} - \bs(\dvh,\psI) - \sum_{\P\in\Th}\Big(\asPh(\uvI-\uv_{\pi},\dvh) + \asP (\uv_{\pi},\dvh) \Big)
    \\[0.5em]
    &= \bil{\fvh}{\dvh} - \bs(\dvh,\psI) - \sum_{\P\in\Th}\asPh(\uvI-\uv_{\pi},\dvh)
    \\[0.5em]
    &\qquad
    - \sum_{\P\in\Th}\Big( \asP (\uv_{\pi}-\uv,\dvh) + \asP(\uv,\dvh) \Big).
  \end{align*}
  We use~\eqref{eq:asP:def},~\eqref{eq:stokes:var:A}, and rearrange
  the terms:
  \begin{align*}
    \Cs\snorm{\dvh}{1,\Omega}^2
    &\leq \bil{\fvh}{\dvh} - \bs(\dvh,\psI) - \sum_{\P\in\Th}\asPh(\uvI-\uv_{\pi},\dvh) - \sum_{\P\in\Th}\asP (\uv_{\pi}-\uv,\dvh)
    \\[0.5em]
    &\qquad
    - \as(\uv,\dvh)
    \\[0.5em]
    &= \bil{\fvh}{\dvh} - \bs(\dvh,\psI) - \sum_{\P\in\Th}\asPh(\uvI-\uv_{\pi},\dvh) - \sum_{\P\in\Th}\asP (\uv_{\pi}-\uv,\dvh)
    \\[0.5em]
    &\qquad
    - \Big( \scal{\fv}{\dvh} - \bs(\dvh,\ps) \Big)
    \\[0.5em]
    &= \Big[ \bil{\fvh}{\dvh} - \scal{\fv}{\dvh} \Big]
    +  \Big[ \bs(\dvh,\ps)    - \bs(\dvh,\psI)   \Big]
    \\[0.5em]
    &\qquad
    + \Bigg[ - \sum_{\P\in\Th}\asPh(\uvI-\uv_{\pi},\dvh) - \sum_{\P\in\Th}\asP (\uv_{\pi}-\uv,\dvh) \Bigg]
    \nonumber\\[0.4em]
    &
    = \big[\TERM{R1}\big] + \big[\TERM{R2}\big] + \big[\TERM{R3}\big].
  \end{align*}
  At this point, we continue estimating separately the three terms
  $\TERM{R1}$, $\TERM{R2}$, and $\TERM{R3}$, which are
  identified in the last step of the above inequality chain by the
  square brackets.

  \medskip
  \medskip
  Term $\TERM{R1}$ is bounded from above by multiplying and dividing
  by $\dvh$, and then taking the supremum
  \begin{align*}
    \ABS{\TERM{R1}}
    &= \ABS{\bil{\fvh}{\dvh} - \scal{\fv}{\dvh}}
    = \frac{ \ABS{\bil{\fvh}{\dvh} - \scal{\fv}{\dvh}} }{ \snorm{\dvh}{1,\Omega} }\,\snorm{\dvh}{1,\Omega}
    \\[0.5em]
    &\leq \left[\sup_{\vvh\in\Vvhk\backslash\{\zerov\}}\frac{ \ABS{\bil{\fvh}{\vvh} - \scal{\fv}{\vvh}} }{ \snorm{\vvh}{1,\Omega} }\right]\,\snorm{\dvh}{1,\Omega}.
  \end{align*}
  Term $\TERM{R2}$ is bounded from above by
  applying the Cauchy-Schwarz inequality to find that
  \begin{align*}
    \ABS{\TERM{R2}}
    &
    = \ABS{\bs(\dvh,\ps-\psI)}
    = \ABS{\scal{\DIV\dvh}{\ps-\psI}}
    \leq \norm{\DIV\dvh}{0,\Omega}\,\norm{\ps-\psI}{0,\Omega}
    \\[0.5em] &
    \leq \snorm{\dvh}{1,\Omega}\,\norm{\ps-\psI}{0,\Omega}.
  \end{align*}
  Finally, term $\TERM{R3}$ is bounded from above by using the local
  continuity of $\asPh(\cdot,\cdot)$,
  cf.~\eqref{eq:ash:local:continuity}, and $\asP(\cdot,\cdot)$
  to find that
  \begin{align*}
    \ABS{\TERM{R3}}
    &= \ABS{ \sum_{\P\in\Th}\Big( \asPh(\uvI-\uv_{\pi},\dvh) + \asP(\uv_{\pi}-\uv,\dvh) \Big) }
    \nonumber\\[0.5em]
    &\leq\sum_{\P\in\Th}\Big(\ABS{\asPh(\uvI-\uv_{\pi},\dvh)} + \ABS{\asP(\uv_{\pi}-\uv,\dvh)}\Big)
    \nonumber\\[0.5em]
    &\leq\sum_{\P\in\Th}\Big(
    \alpha^*\snorm{\uvI-\uv_{\pi}}{1,\P} +
    \snorm{\uv-\uv_{\pi}}{1,\P}
    \Big)\snorm{\dvh}{1,\P}.
  \end{align*}
  Then, we add and subtract $\uv$ in the first summation argument,
  and, in the last step, we introduce the broken seminorm
  $\snorm{\cdot}{1,\hh}^2=\sum_{\P\in\Th}\snorm{\cdot}{1,\P}^2$ to
  find that
  \begin{align*}
    \ABS{\TERM{R3}}
    &\leq
    \sum_{\P\in\Th}\Big(
    \alpha^*\snorm{\uvI-\uv}{1,\P} +
    (1+\alpha^*)\snorm{\uv-\uv_{\pi}}{1,\P}
    \Big)\,\snorm{\dvh}{1,\P}
    \nonumber\\[0.5em]
    &\leq\Big(\,
    \alpha^*\snorm{\uv-\uv_{I}}{1,\Omega} + (1+\alpha^*)\snorm{\uv-\uv_{\pi}}{1,\hh}
    \,\Big)
    \,\snorm{\dvh}{1,\Omega}.
  \end{align*}

  We are left to derive a bound for the pressure term.
  To this end, we consider $\ssh=\psh-\psI$.
  Note that $\psh$ is the solution
  of~\eqref{eq:stokes:vem:A}-\eqref{eq:stokes:vem:B} and the integral
  of $\psI$ on $\Omega$ is zero because $\restrict{\psI}{\P}$ has the
  same degrees of freedom, i.e., polynomial moments, of
  $\restrict{\ps}{\P}$, so that
  \begin{align*}
    \int_{\Omega}\psI\dV
    = \sum_{\P\in\Th}\int_{\P}\psI\dV
    = \sum_{\P\in\Th}\int_{\P}\ps\dV
    = \int_{\Omega}\ps\dV=0.
  \end{align*}
  Henceforth, the global interpolant $\psI$, which is built from the
  local interpolant that is such that
  $\restrict{\psI}{\P}=\big(\restrict{\ps}{\P}\big)_{\INTP}\in\Qshkk(\P)$,
  belongs to $\Qshkk$.
  We start the derivation of a bound for $\norm{\ssh}{0,\Omega}$ from
  the inf-sup condition~\eqref{eq:discrete:infsup}, which implies that
  for every $\ssh\in\Qshkk$, there exists a vector field
  $\vvh\in\Vvhk$ such that
  \begin{align*}
    \beta\norm{\ssh}{0,\Omega}\snorm{\vvh}{1,\Omega}
    \leq \bsh(\vvh,\ssh).
  \end{align*}
  Then, we split $\ssh=\psh-\psI$, we use
  equation~\eqref{eq:stokes:vem:A} and add~\eqref{eq:stokes:var:A} and
  we find that
  \begin{align*}
    &\bsh(\vvh,\ssh)
    =   \bsh(\vvh,\psh) - \bsh(\vvh,\psI)
    = -\ash(\uvh,\vvh) + \bil{\fvh}{\vvh} - \bsh(\vvh,\psI)
    \\[0.5em]
    &\qquad
    = -\ash(\uvh,\vvh) + \bil{\fvh}{\vvh} - \bsh(\vvh,\psI)
    + \Big( \as(\uv,\vvh) + \bs(\vvh,\ps) - \scal{\fv}{\vvh} \big).
  \end{align*}
  We rearrange the terms, we use~\eqref{eq:asP:def}
  and~\eqref{eq:ash:def}, we add and substract $\uv_{\pi}$, we note
  again that $\bsh(\vvh,\psI)=\bs(\vvh,\psI)$,
  use~\eqref{eq:consistency} with $\qv=\uv_{\pi}$ and we find that
  \begin{align}
    \beta\norm{\ssh}{0,\Omega}\snorm{\vvh}{1,\Omega}
    &\leq \sum_{\P\in\Th}\Big( \asP(\uv,\vvh) - \asPh(\uvh,\vvh) \Big) + \bil{\fvh}{\vvh} - \scal{\fv}{\vvh}
    \nonumber\\
    &\qquad+ \bs(\vvh,\ps) - \bsh(\vvh,\psI)
    \nonumber\\[0.5em]
    &=
    \Big[ \bil{\fvh}{\vvh} - \scal{\fv}{\vvh} \Big] + \Big[ \bs(\vvh,\ps) - \bs(\vvh,\psI) \Big]
    \nonumber\\
    &\qquad+  
    \sum_{\P\in\Th}\Big( \asP(\uv-\uv_{\pi},\vvh) - \asPh(\uvh-\uv_{\pi},\vvh) \Big)
    \nonumber\\[0.5em]
    &= \big[\TERM{R4}\big] + \big[\TERM{R5}\big] + \big[\TERM{R6}\big],
  \end{align}
  where, again, terms $\TERM{R4}$, $\TERM{R5}$, and $\TERM{R6}$
  are identified by the square brackets.
  As for the term $\dvh$, we continue estimating each term separately.

  \medskip
  Term $\TERM{R4}$ is bounded from above as we bounded term
  $\TERM{R1}$ with $\vvh$ instead of $\dvh$
  \begin{align*}
    \ABS{\TERM{R4}}
    = \ABS{\bil{\fvh}{\vvh} - \scal{\fv}{\vvh}}
    \leq \left[ \sup_{\vvh\in\Vvhk\backslash\{\zerov\}} \frac{ \ABS{\bil{\fvh}{\vvh} - \scal{\fv}{\vvh}} }{ \snorm{\vvh}{1,\Omega} }\right]\,\snorm{\vvh}{1,\Omega}.
  \end{align*}
  Term $\TERM{R5}$ is bounded from above as we bounded term
  $\TERM{R2}$ with $\vvh$ instead of $\dvh$
  \begin{align*}
    \ABS{\TERM{R5}}
    \leq \snorm{\vvh}{1,\Omega}\,\norm{\psI-\ps}{0,\Omega}.
  \end{align*}
  Term $\TERM{R6}$ is bounded from above as we bounded term
  $\TERM{R3}$ with $\vvh$ instead of $\dvh$
  \begin{align*}
    \ABS{\TERM{R6}}
    &\leq\Big(
    \alpha^*\snorm{\uv-\uv_{I}}{1,\Omega} +
    (1+\alpha^*)\snorm{\uv-\uv_{\pi}}{1,\hh}
    \Big)\,\snorm{\vvh}{1,\Omega},
  \end{align*}

  \medskip
  Eventually, we obtain the energy estimate of
  inequality~\eqref{eq:theo:abstract} by adding and substracting
  $\uvI$ and $\psI$ in the two terms of the left-hand side
  of~\eqref{eq:theo:abstract} and using the triangle inequalities
  \begin{align}
    \snorm{\uv-\uvh}{1,\Omega} \leq \snorm{\uv-\uvI}{1,\Omega} + \snorm{\uvI-\uvh}{1,\Omega}, \\[0.2em]
    \norm {\ps-\psh}{0,\Omega} \leq \norm {\ps-\psI}{0,\Omega} + \snorm{\psI-\psh}{0,\Omega}.
  \end{align}
  We estimate the first term in the right-hand side of the two
  inequalities above by using standard estimates of the interpolation
  errors provided by Lemma~\ref{lemma:projection:error} and the second
  term by substituting the bounds we previously derived for the terms
  $\TERM{R1}-\TERM{R6}$.
\end{proof}

\begin{remark}
  In the proof of Theorem~\ref{theorem:H1:abstract}, we need to deal
  with a pressure dependent term in the estimate of the velocity error,
  e.g., term~\TERM{R2}.
  A different approach permits the derivation of an estimate of the
  velocity error that does not require to cope with pressure terms,
  see, e.g.,~\cite[Theorem~4.6]{BeiraodaVeiga-Lovadina-Vacca:2017}.
\end{remark}

An error estimate follows immediate from
Theorem~\ref{theorem:H1:abstract}.
We state it in the following corollary.

\begin{corollary}[Error estimate]
  \label{corollary:error:estimate}
  Let $(\uvh,\psh)\in\Vvhk\times\Qshkk$, $k\geq2$, be the virtual
  element solution fields approximating
  $\uv\in\big[\HS{s+1}(\Omega)\cap\HONEzr(\Omega)\big]^2$ and
  $\ps\in\HS{s}(\Omega)\cap\LTWOzr(\Omega)$, and take
  $\fv\in\big[\HS{s}(\Omega)\big]^2$ for $1\leq\ss\leq\ks$.
  Then, under the assumptions of Theorem~\ref{theorem:H1:abstract}, it
  holds that
  \begin{align}
    \label{eq:the:H1:estimates}
    \snorm{\uv-\uvh}{1,\Omega} + \norm{\ps-\psh}{0,\Omega}
    \leq \Cs\hh^{s}\Big(
    \norm{\uv}{s+1,\Omega} + \norm{\ps}{s,\Omega} + \norm{\fv}{s,\Omega} 
    \Big)
  \end{align}
  for some real, strictly positive constant $\Cs$ independent of
  $\hh$.
\end{corollary}
\begin{proof}
  Estimate~\eqref{eq:the:H1:estimates} follows from a straightforward
  application of the results of Lemmas~\ref{lemma:interpolation:error}
  and~\ref{lemma:projection:error}, and
  estimates~\eqref{eq:fv:bound:0}-\eqref{eq:fv:bound:1} to the
  right-hand side of~\eqref{eq:theo:abstract}.
\end{proof}

\begin{remark}
  The abstract convergence result of Theorem~\ref{theorem:H1:abstract}
  and the error estimate of Corollary~\ref{corollary:error:estimate}
  hold for $k\geq2$.
  On the one hand, this limitation is a consequence of
  Theorem~\ref{theorem:discrete:infsup}, which states that the inf-sup
  stability holds for $k\geq2$.
  On the other hand, a convergence result for $k=1$ does not hold in
  general since we know that this scheme can be unstable on some
  meshes, including triangular and quadrilateral meshes.
  An extension to $k=1$ is however possible by requiring the mesh
  elements to satisfy stronger conditions than Assumptions
  $\textbf{(M1)}-\textbf{(M2)}$ and including edge bubble functions in
  the formulation of the method.
  An example of this approach is provided by the work of
  Reference~\cite{BeiraodaVeiga-Lipnikov:2010}.
\end{remark}

\section{Numerical experiments}
\label{sec:numerical}


\begin{figure}
  \centering
  \begin{tabular}{cccccc}
    \hspace{-0.42cm}\includegraphics[scale=0.2]{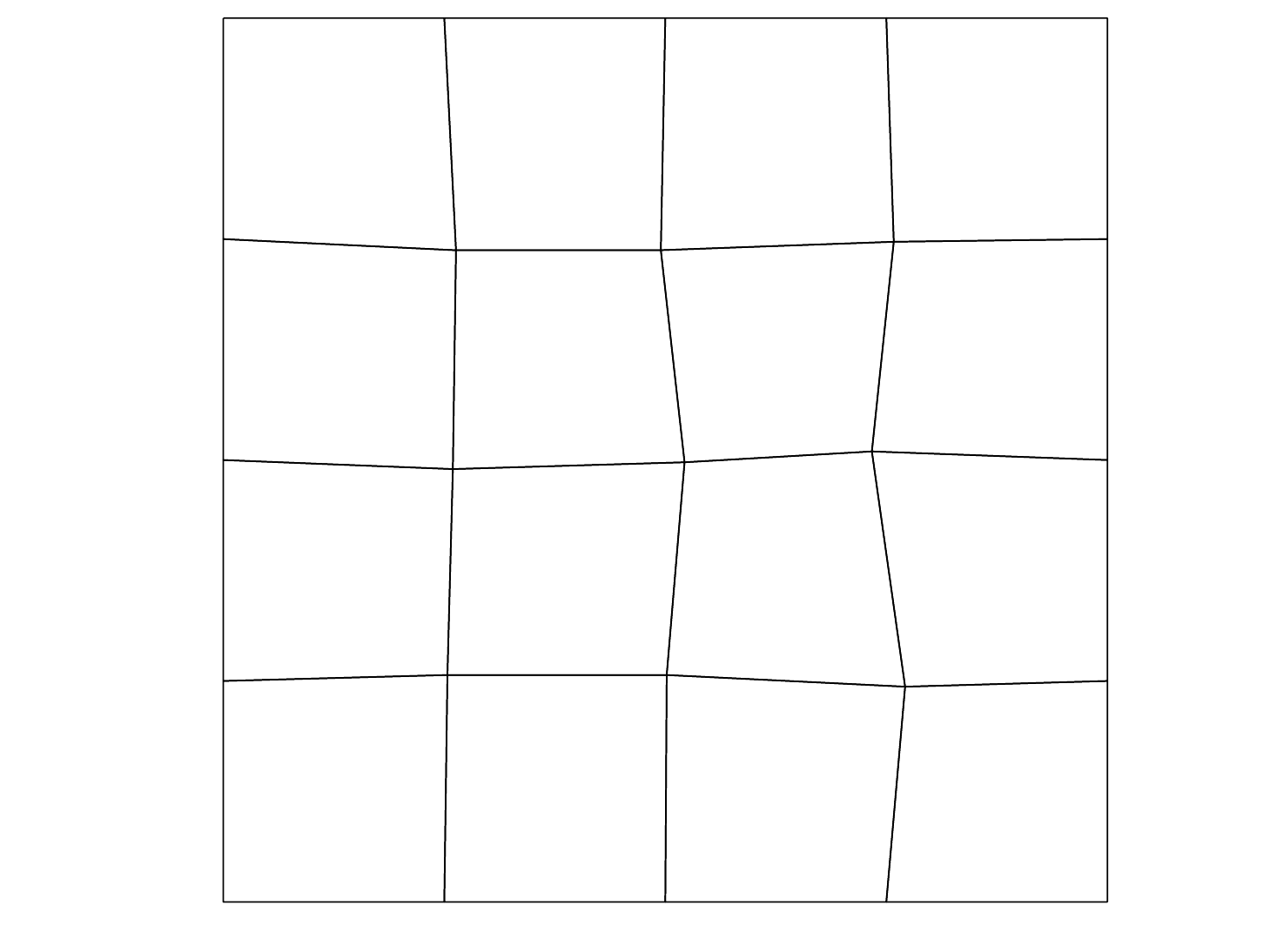} & 
    \hspace{-0.42cm}\includegraphics[scale=0.2]{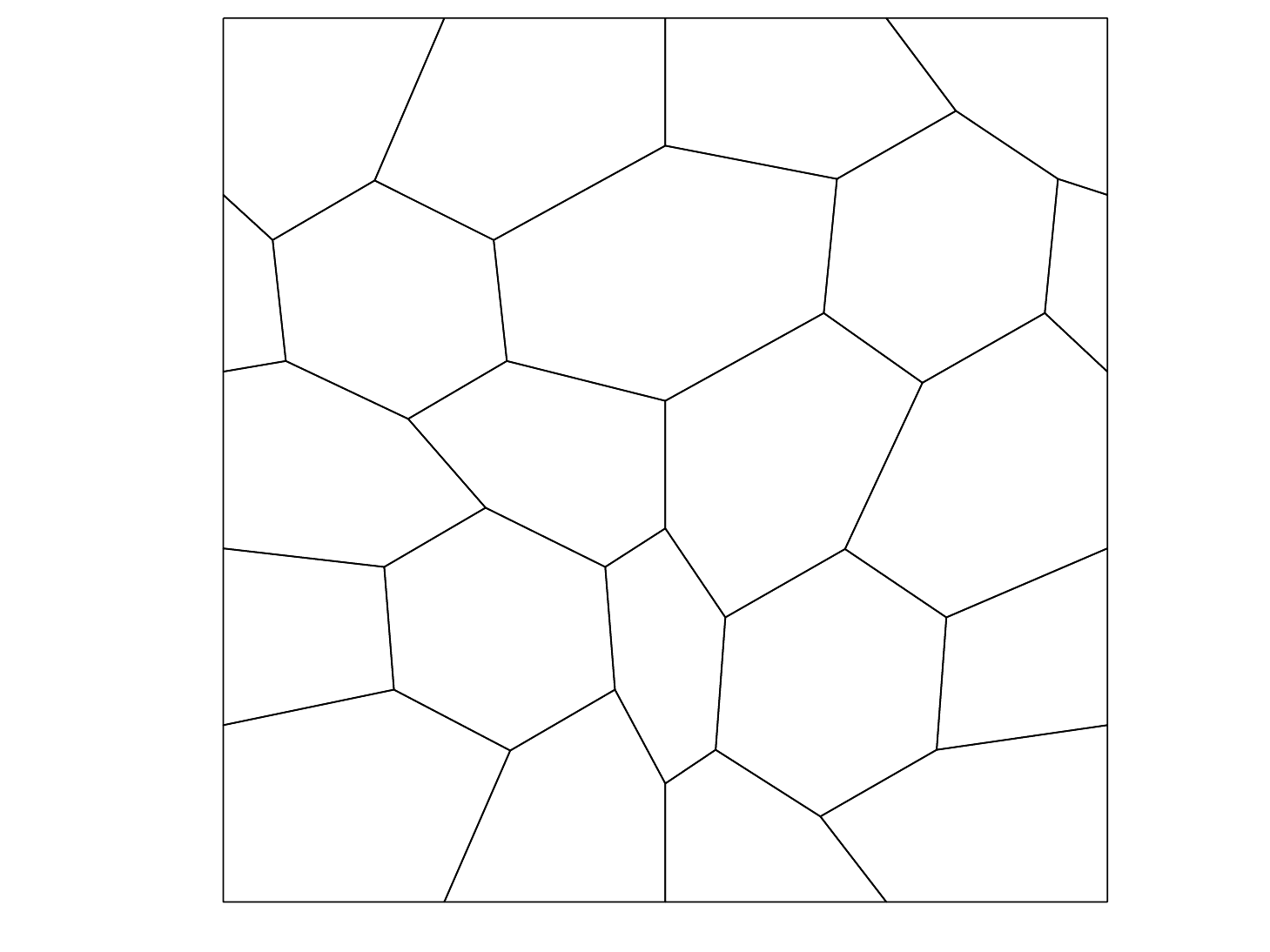} & 
    \hspace{-0.42cm}\includegraphics[scale=0.2]{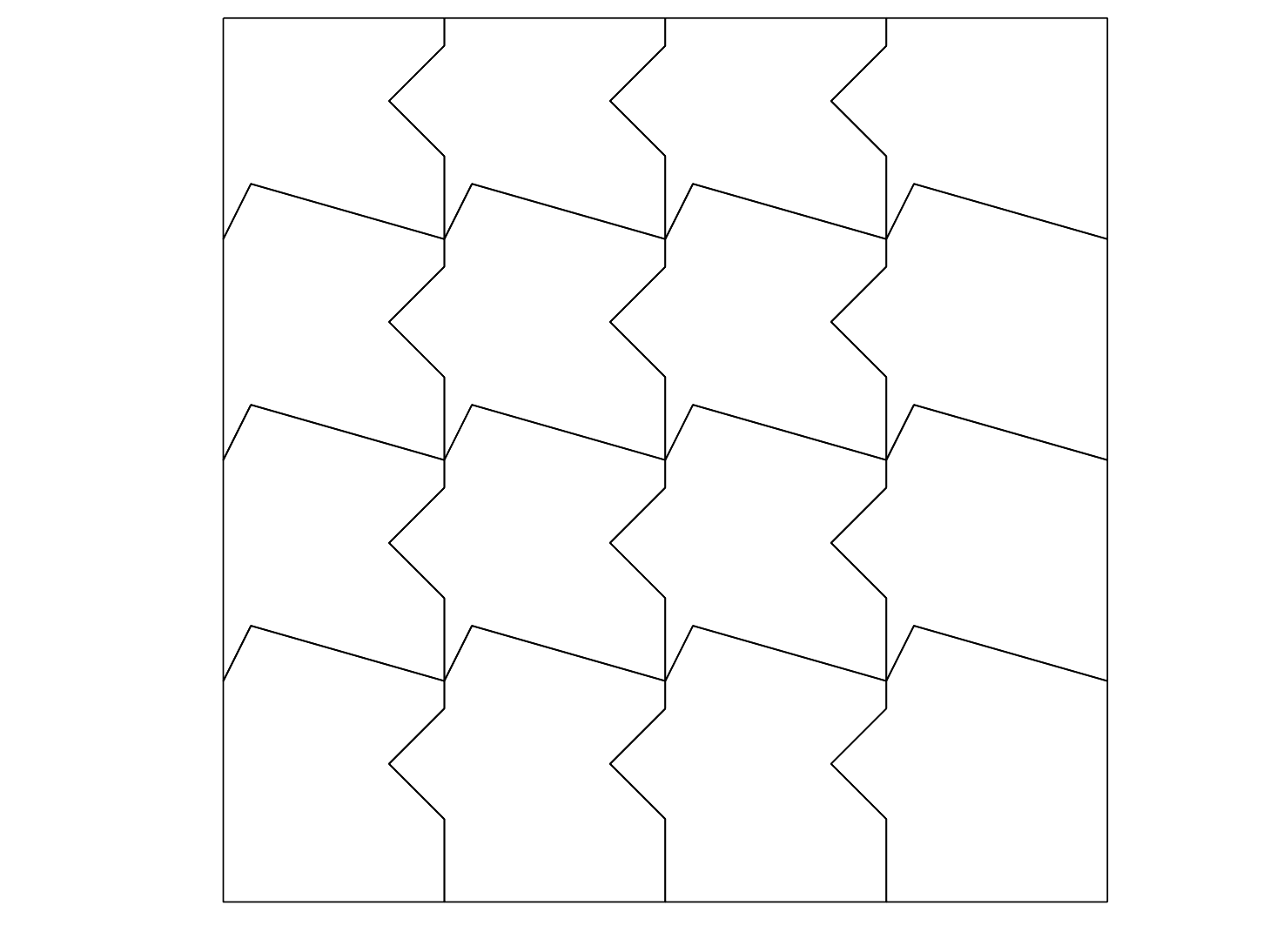} & 
    \hspace{-0.42cm}\includegraphics[scale=0.2]{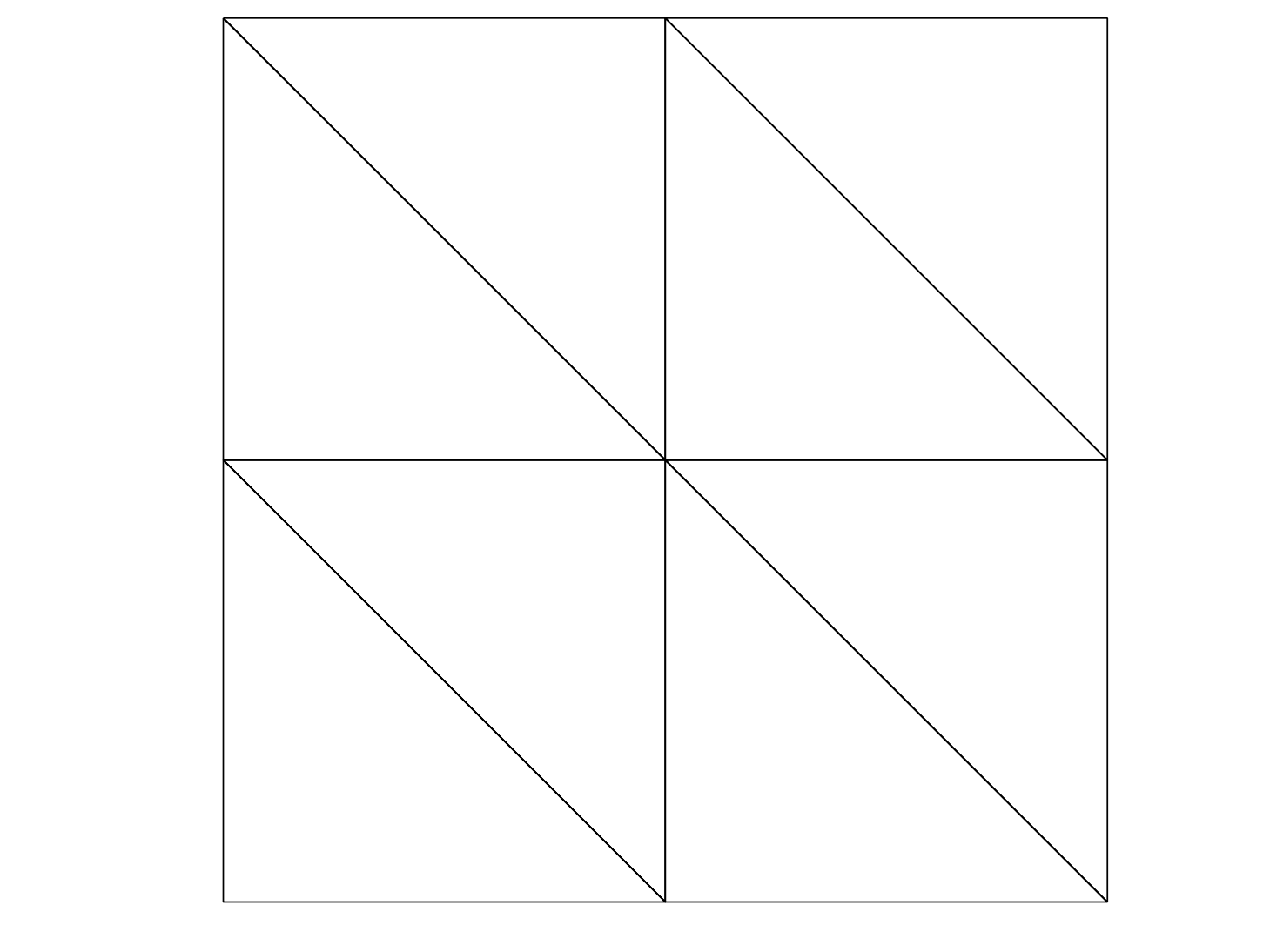} & 
    \hspace{-0.42cm}\includegraphics[scale=0.2]{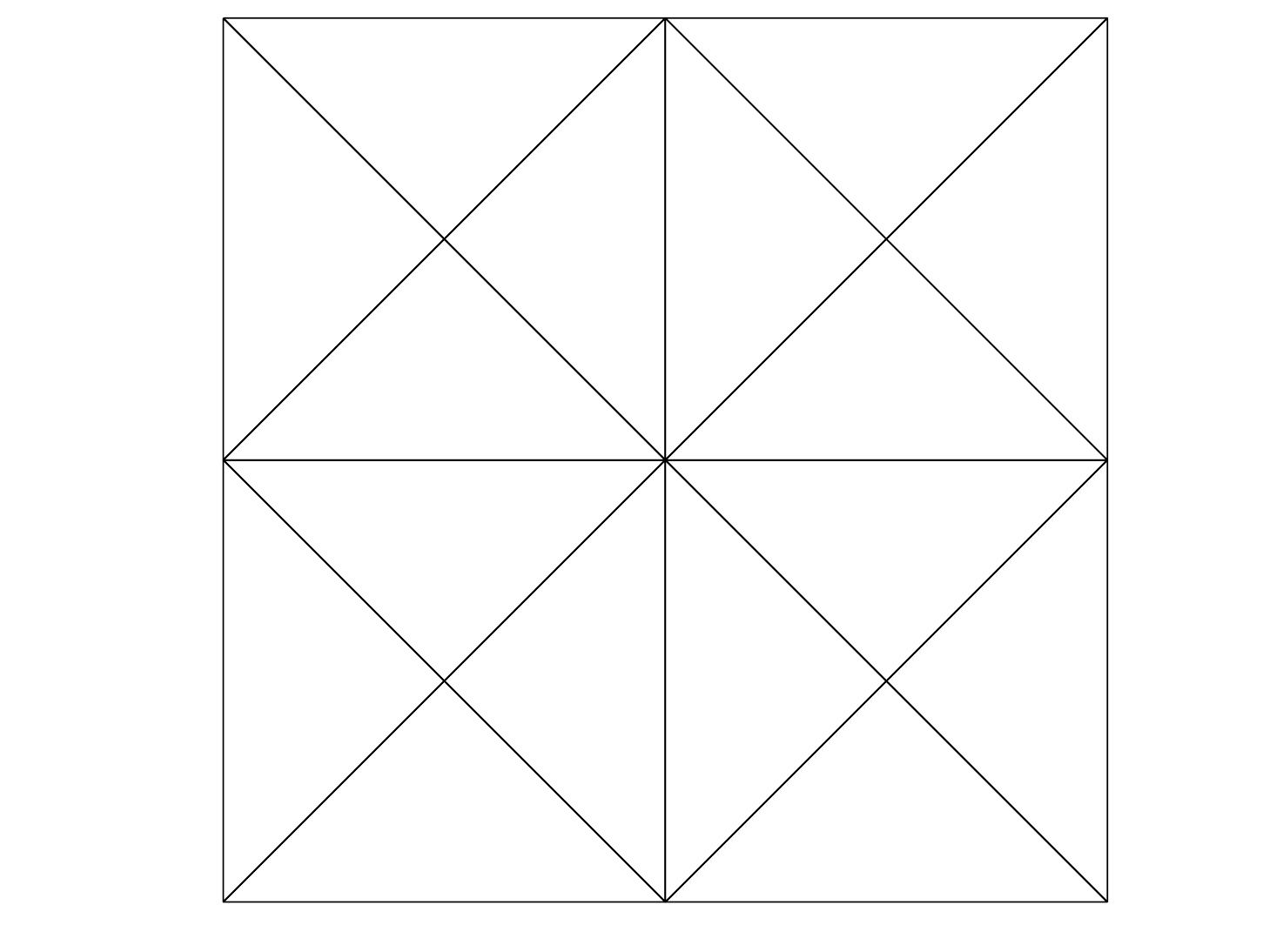} & 
    \hspace{-0.42cm}\includegraphics[scale=0.2]{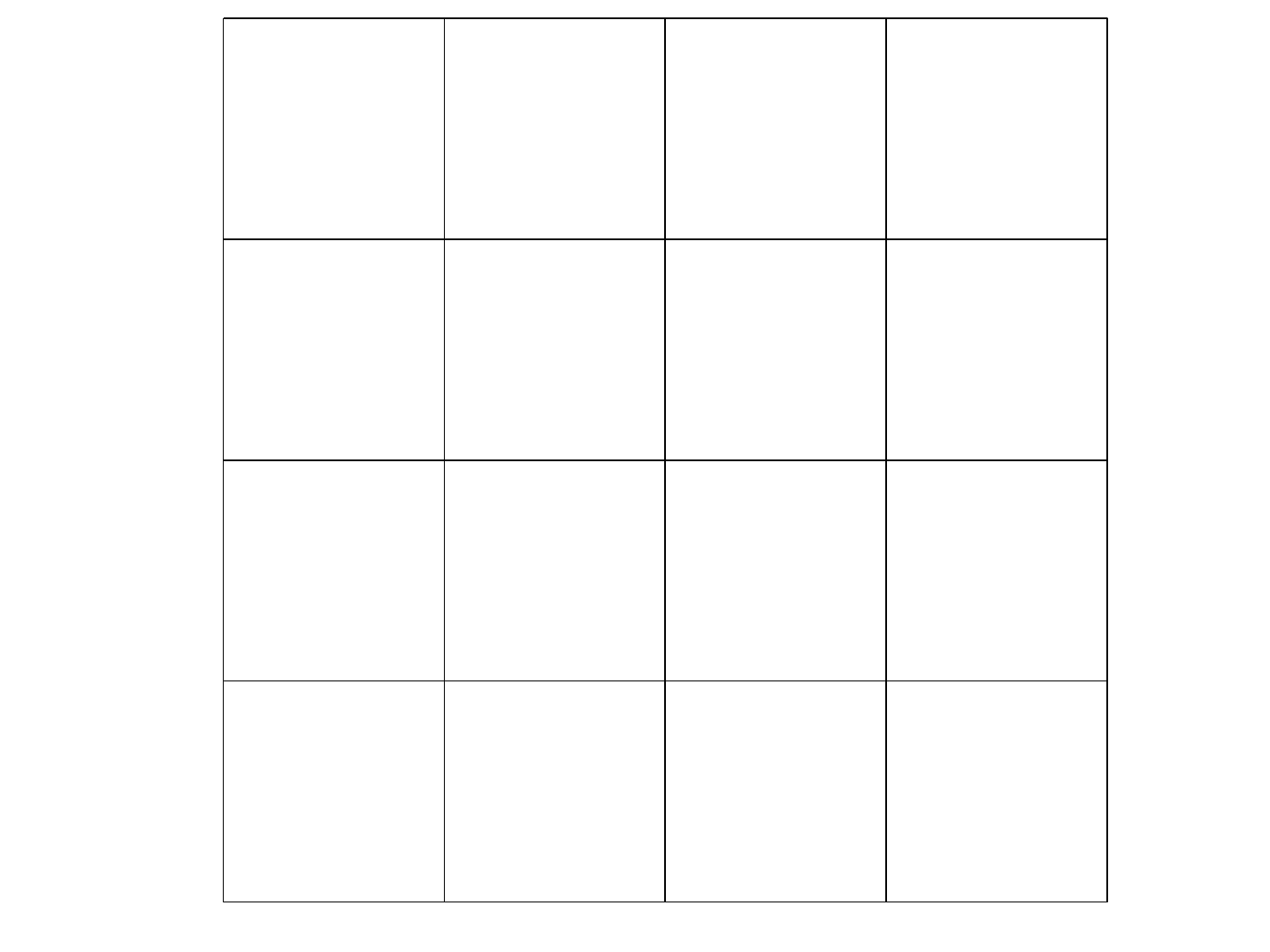}   
    \\[1em]
    \hspace{-0.42cm}\includegraphics[scale=0.2]{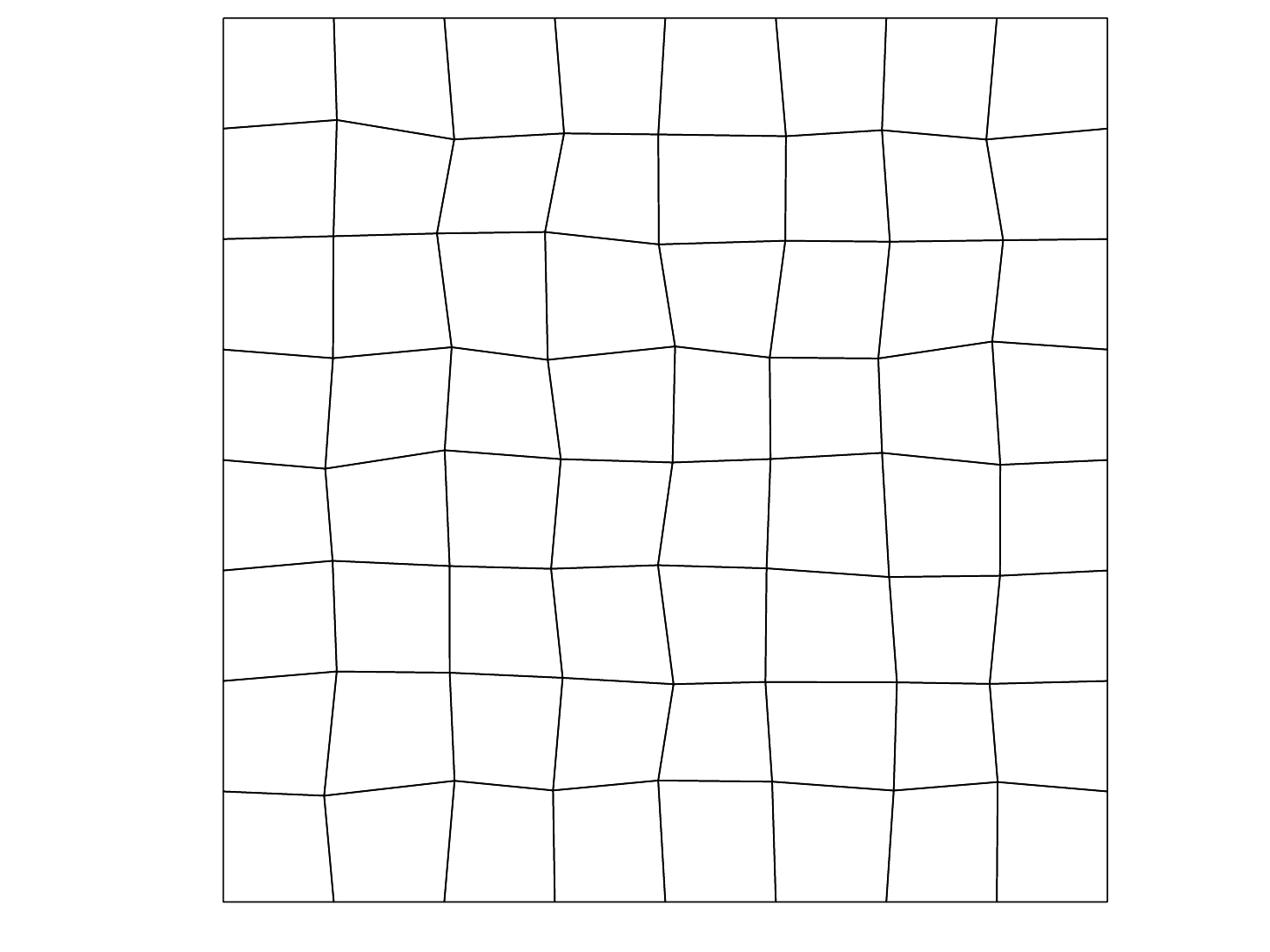} & 
    \hspace{-0.42cm}\includegraphics[scale=0.2]{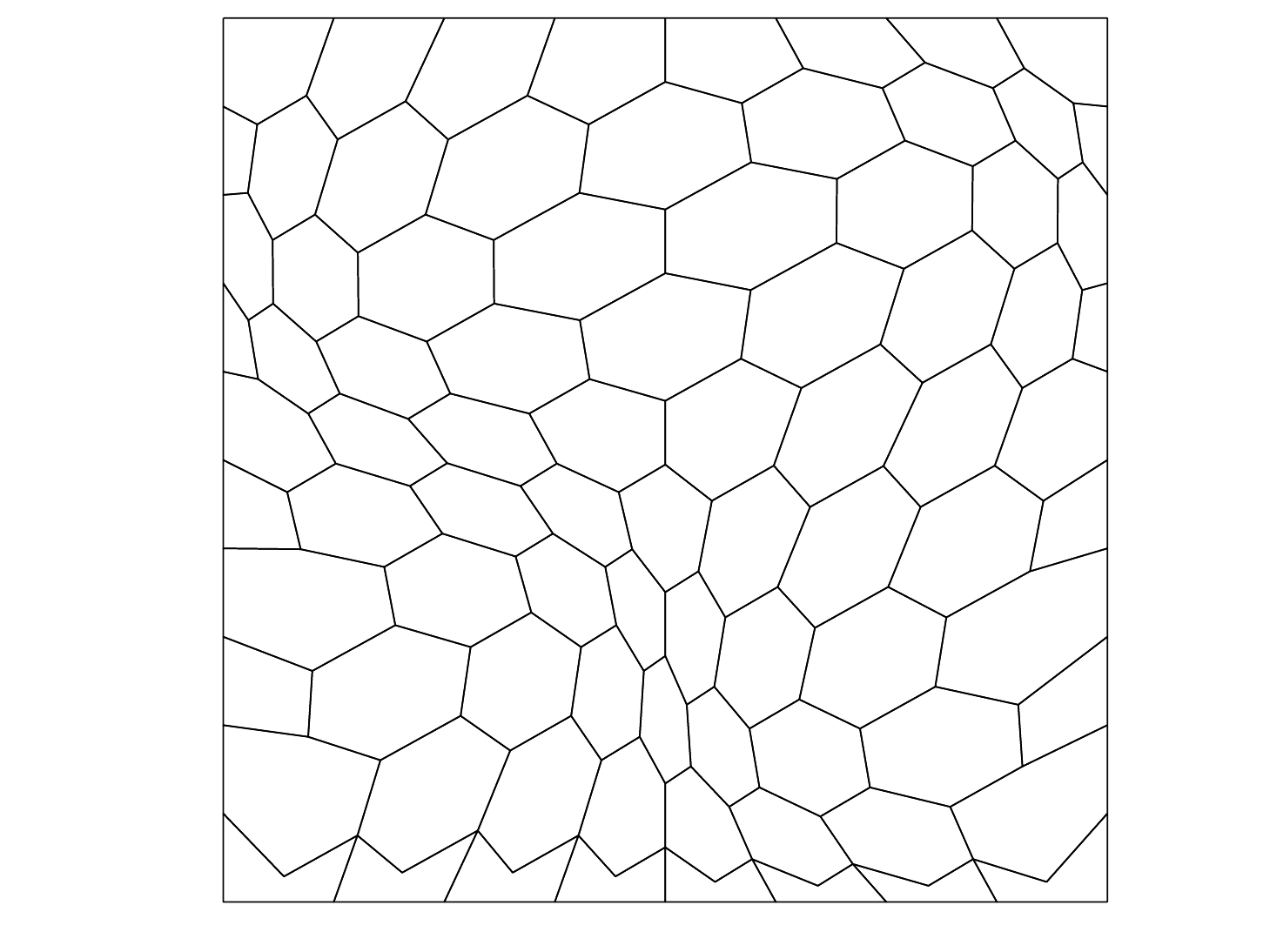} & 
    \hspace{-0.42cm}\includegraphics[scale=0.2]{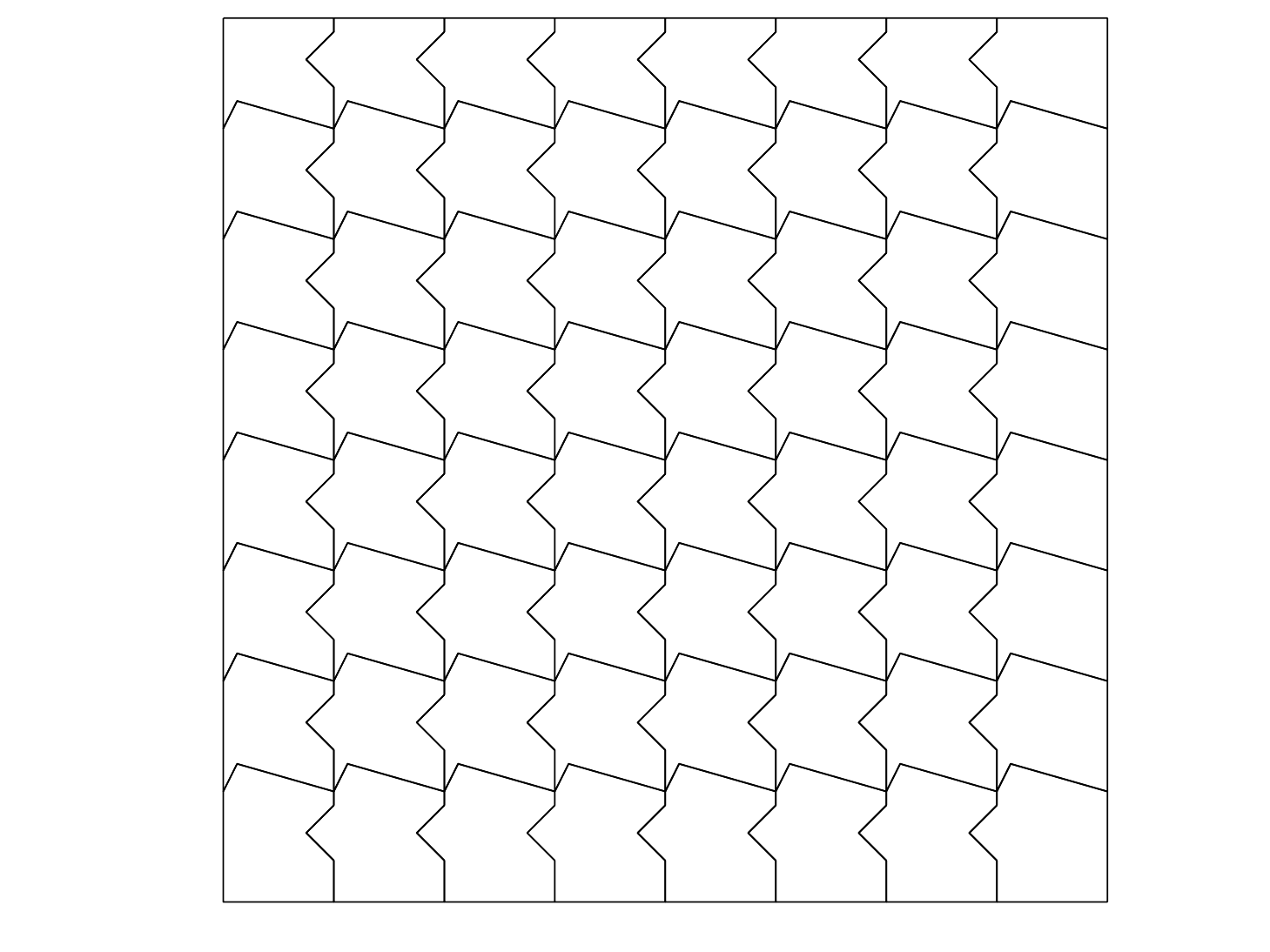} & 
    \hspace{-0.42cm}\includegraphics[scale=0.2]{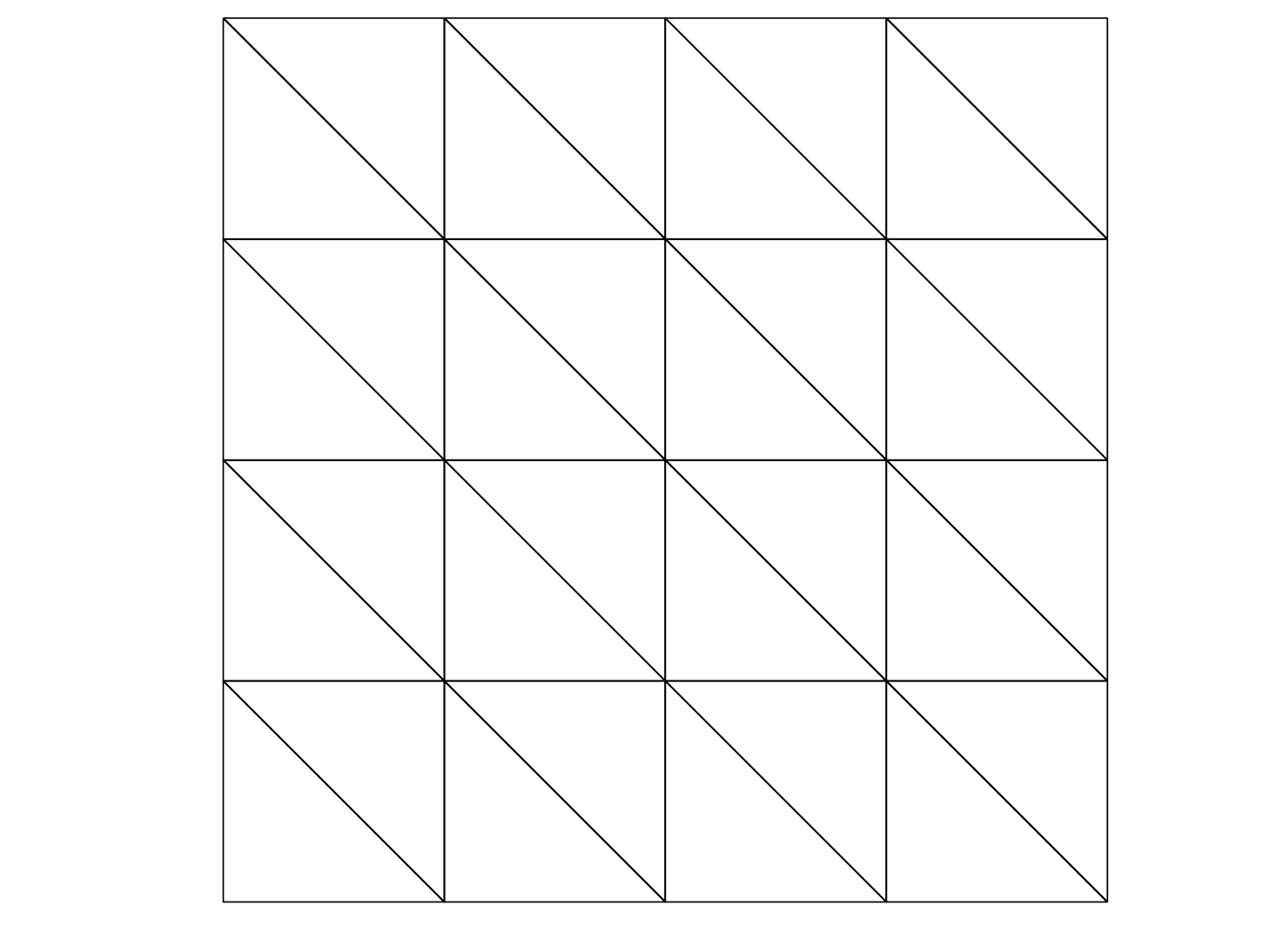} & 
    \hspace{-0.42cm}\includegraphics[scale=0.2]{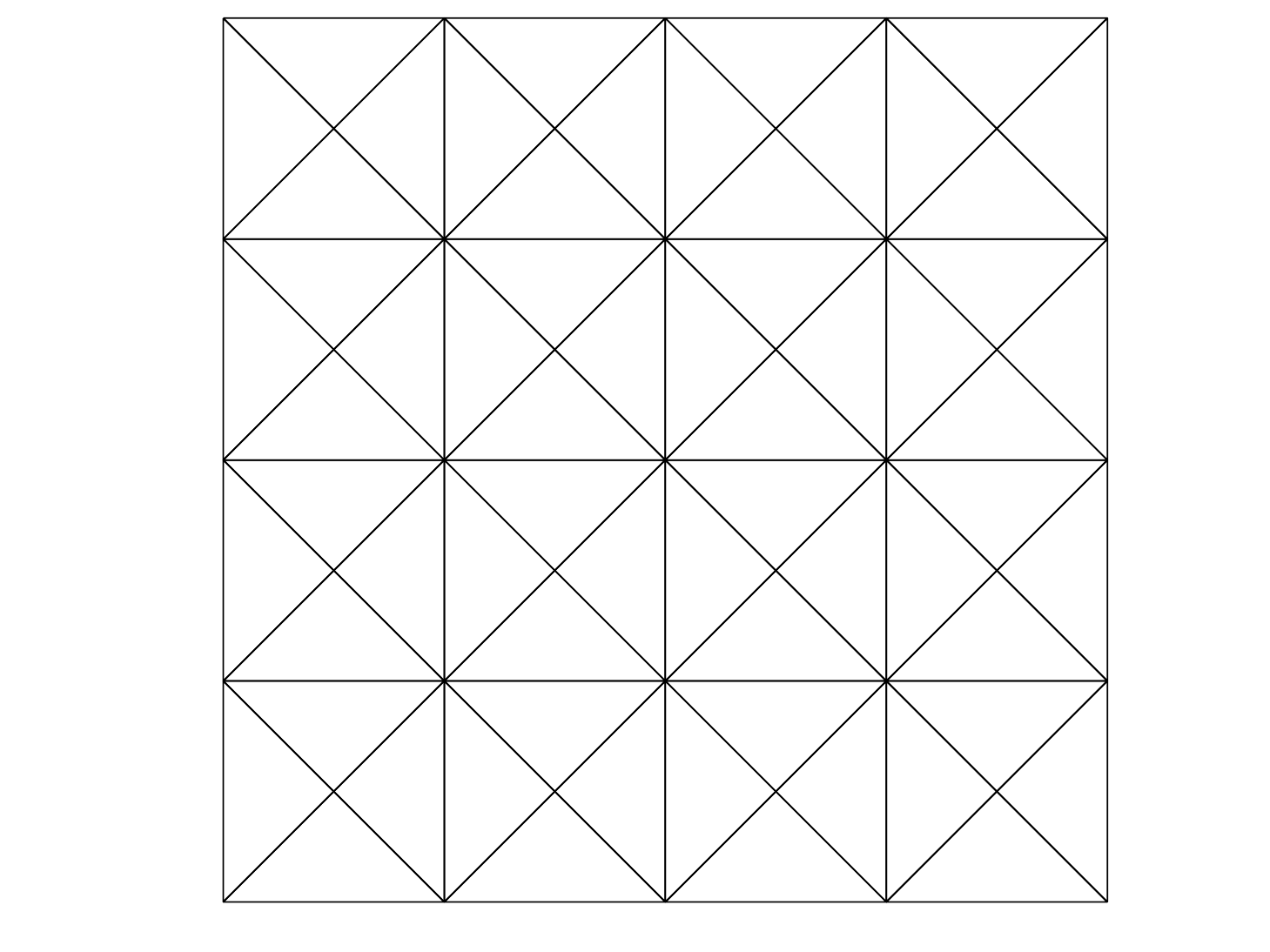} & 
    \hspace{-0.42cm}\includegraphics[scale=0.2]{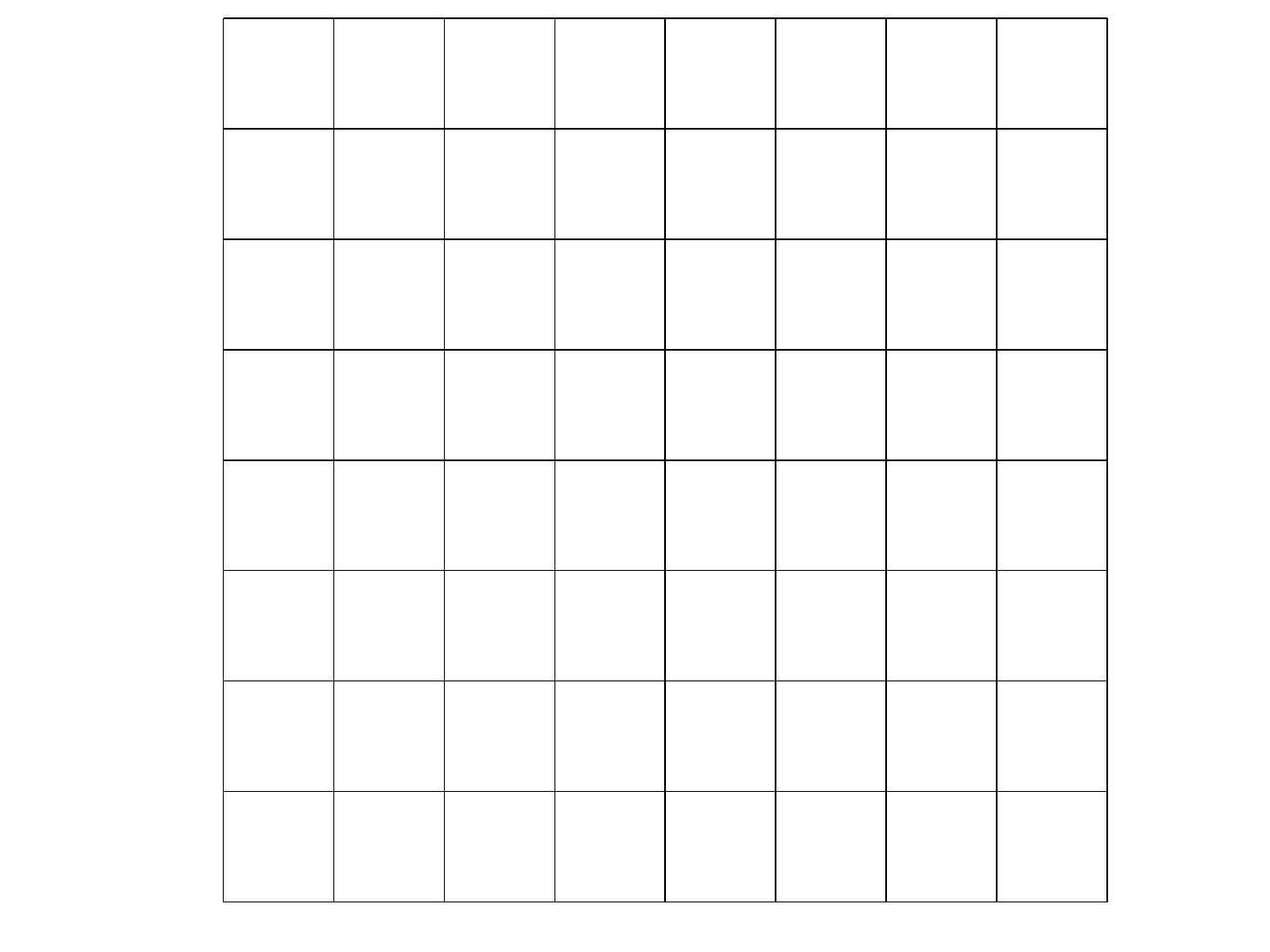}   
    \\[1em]    
    \hspace{-2mm}\text{(a)} & \hspace{-2mm}\text{(b)} & \hspace{-2mm}\text{(c)} &
    \hspace{-2mm}\text{(d)} & \hspace{-2mm}\text{(e)} & \hspace{-2mm}\text{(f)}
  \end{tabular}
  \caption{Base meshes (top row) and first refinement meshes (bottom
    row) of the six mesh families used in this section:
    $(a)$ randomly quadrilateral meshes;
    $(b)$ general polygonal meshes; 
    $(c)$ concave element meshes;
    $(d)$ diagonal triangle meshes;
    $(e)$ criss-cross triangle meshes;
    $(f)$ square meshes.
    }
  \label{fig:Meshes}
\end{figure}
We assess the convergence property of the  virtual element
formulation considered in this paper by numerically solving problem
\eqref{eq:stokes:var:A}-\eqref{eq:stokes:var:B} on the computational
domain $\Omega=[0,1]\times[0,1]$ partitioned by the six mesh families
of Figure~\ref{fig:Meshes}.
Dirichlet boundary conditions and source terms are set accordingly to
the manufactured solution $\uv=(\us_x,\us_y)^T$ and $\ps$ given by
\begin{align*}
  \us_x(x,y) &=  \cos{(2\pi\xs)}\sin{(2\pi\ys)},\\
  \us_y(x,y) &= -\sin{(2\pi\xs)}\cos{(2\pi\ys)},\\
  \ps  (x,y) &= e^{\xs+\ys}-(e-1)^2.
\end{align*}
On any set of refined meshes, we measure the $\HONE$ relative error
for the velocity vector field by applying the formula
\begin{align}
  \text{error}_{\HONE(\Omega)}(\uv) = \frac{\snorm{\uv-\Piz{k}\uvh}{1,\hh}}{\snorm{\uv}{1,\Omega}}
  \approx \dfrac{\snorm{\uv-\uvh}{1,\Omega}}{\snorm{\uv}{1,\Omega}},
  \label{eq:error:H1:velocity}
\end{align}
and the $\LTWO$ relative error by applying the formula
\begin{align}
  \text{error}_{\LTWO(\Omega)}(\uv)
  = \dfrac{\norm{ \uv-\Piz{k}\uvh}{0,\Omega}}{\norm{\uv}{0,\Omega}}
  \approx \dfrac{ \norm{\uv-\uvh}{0,\Omega} }{\norm{\uv}{0,\Omega}}.
  \label{eq:error:L2:velocity}
\end{align}
For the pressure scalar field we measure the $\LTWO(\Omega)$ relative
error by applying the formula
\begin{align}
  \text{error}_{\LTWO(\Omega)}(\ps)
  = \dfrac{\norm{\ps-\psh}{0,\Omega}}{\norm{\ps}{0,\Omega}}.
  \label{eq:error:L2:pressure}
\end{align}

Discretization of the bilinear form $a_h(\cdot,\cdot)$ and
$b_h(\cdot,\cdot)$ yields two matrices $\matA$ and $\matB$,
respectively, while discretization of the right-hand side gives the
vector $\mathbf{f}$.

To write the formal expression of such matrices we introduce the
canonical shape functions $\Phi_j$ of the virtual element space
$\Vvhk(\P)$, and the basis functions $\mathbf{m}_{\alpha}$ of the
local polynomial space $\big[\PS{\ell}(\P)\big]^2$, wheer $\ell$ can
be equal to $k-1$ or $k$.
We can assume that space $\PS{\ell}(\P)$ is the span of the finite set
of \emph{scaled monomials of degree up to $\ell$}, that are given by
\begin{align*}
  \calM_{\ell}(\omega) =
  \bigg\{\,
  \left( \frac{\xv-\xv_{\omega}}{\hh_{\omega}} \right)^{\alpha}
  \textrm{~with~}\abs{\alpha}\leq\ell
  \,\bigg\},
\end{align*}
where 
\begin{itemize}
\item $\xv_{\P}$ denotes the center of gravity of $\P$ and $\hP$ its
  characteristic length, e.g., the cell diameter;
\item $\alpha=(\alpha_1,\alpha_2)$ is the two-dimensional multi-index
  of nonnegative integers $\alpha_i$ with degree
  $\abs{\alpha}=\alpha_1+\alpha_{2}\leq\ell$ and such that
  $\xv^{\alpha}=\xs_1^{\alpha_1}\xs_{2}^{\alpha_{2}}$ for any
  $\xv\in\REAL^{2}$ and
  $\partial^{\abs{\alpha}}\slash{\partial\xv^{\alpha}}=\partial^{\abs{\alpha}}\slash{\partial\xs_1^{\alpha_1}\partial\xs_2^{\alpha_2}}$.
\end{itemize}
Alternatively, we can assume that space $\PS{\ell}(\P)$ is the span of
a set of orthogonalized polynomials built from the scaled monomials by
applying the Gram-Schmidt process in all elements $\P$.

\medskip
Noe, the entries of the global matrix $\matA$ are given by assembling
the local consistency and stability matrices,
\begin{align*}
  \matA= \sum_{\P} \matQ^T\big(\matA_{\P}^C+\matA_{\P}^S \big)\matQ
\end{align*}
where $\matQ$ is the assembling matrix that remaps the local entries
of the elemental matrices $\matA_{\P}^C$ and $\matA_{\P}^S$ into the
global setting of matrix $\matA$, and  
\begin{align*}
  (\matA_{\P}^C)_{i j} &= \int_{\P}\PizP{k-1}\nabla\Phi_i:\PizP{k-1}\nabla\Phi_j\dV,\\[0.5em]
  \matA_{\P}^S        &=(I-\mathbf{\Pi}^{\nabla})^T(I- \mathbf{\Pi}^{\nabla})
\end{align*}
and $\mathbf{\Pi}^{\nabla}$ is the matrix associated to the elliptic
projector.
Likewise, the global matrix $\matB$ has entries given by assembling
the local contributions of the integrals
\begin{align*}
  (\matB_{\P})_{\alpha j } =\int_{\P}
  \mathbf{m}_{\alpha} \PizP{\underline{k}} \DIV \Phi_j \dV.
\end{align*}

\medskip
Denoting with $\mathbf{u}$ and $\mathbf{p}$ the vectors of velocity
and pressure degrees of freedom, respectively, the problem to be
solved reads as:
\begin{equation} \begin{bmatrix} \matA & \matB^T \\
    \matB &  0 
  \end{bmatrix}
  \begin{bmatrix} \mathbf{u} \\ \mathbf{p} \end{bmatrix}=
 \begin{bmatrix} \mathbf{f} \\ 0 \end{bmatrix}
 \label{eq:saddle}
\end{equation}
To solve this system, we eliminate the degrees of freedom
corresponding to the Dirichlet conditions, and impose the additional
condition that $\int_\Omega\ps\dV=0$.

\subsection{General convergence results}

We compare the approximation errors \eqref{eq:error:H1:velocity},
\eqref{eq:error:L2:velocity}, and \eqref{eq:error:L2:pressure}
obtained by applying the numerical scheme to the mesh families $(a)$,
$(b)$, and $(c)$ and setting $\underline{k}=k-1$.
Figure~\ref{fig:h_errorPi0k} shows the approximation errors when the
right-hand side is discretized by using the projection operator
$\Piz{\kb}$ with $\kb=k$.
Figure~\ref{fig:h_errorPikmeno2} shows the approximation errors when
the right-hand side is approximated by using the projection operator
$\Piz{\kb}$ with $\kb=max(0,k-2)$.
When we set $\kb=k$ we observe the optimal convergence rates in all
the norms, so the error for the velocity in the energy norm and in the
$\LTWO$-norm scales as $\mathcal{O}(\hh^k)$ and
$\mathcal{O}(\hh^{k+1})$, and the error of the pressure scales as
$\mathcal{O}(\hh^{k})$.
When we set $\kb=\max(0,k-2)$, we observe the optimal convergence rate
for the velocity approximation for all $k\geq1$ only in the energy
norm.
Instead, the velocity error in the $\LTWO$-norm looses an order of
approximation when $k=2$.
Moreover, an overall better approximation, i.e., smaller errors, are
visible when we select $\kb=k$, which corresponds to the enhanced
virtual element space~\eqref{eq:BF:enhanced-space:def}
Regarding the approximation of the zero-divergence constraint, we note
that the divergence of the velocity is close to the machine precision
for all meshes and $k$ here considered.

\begin{figure}
  \includegraphics[width=\textwidth,clip=]{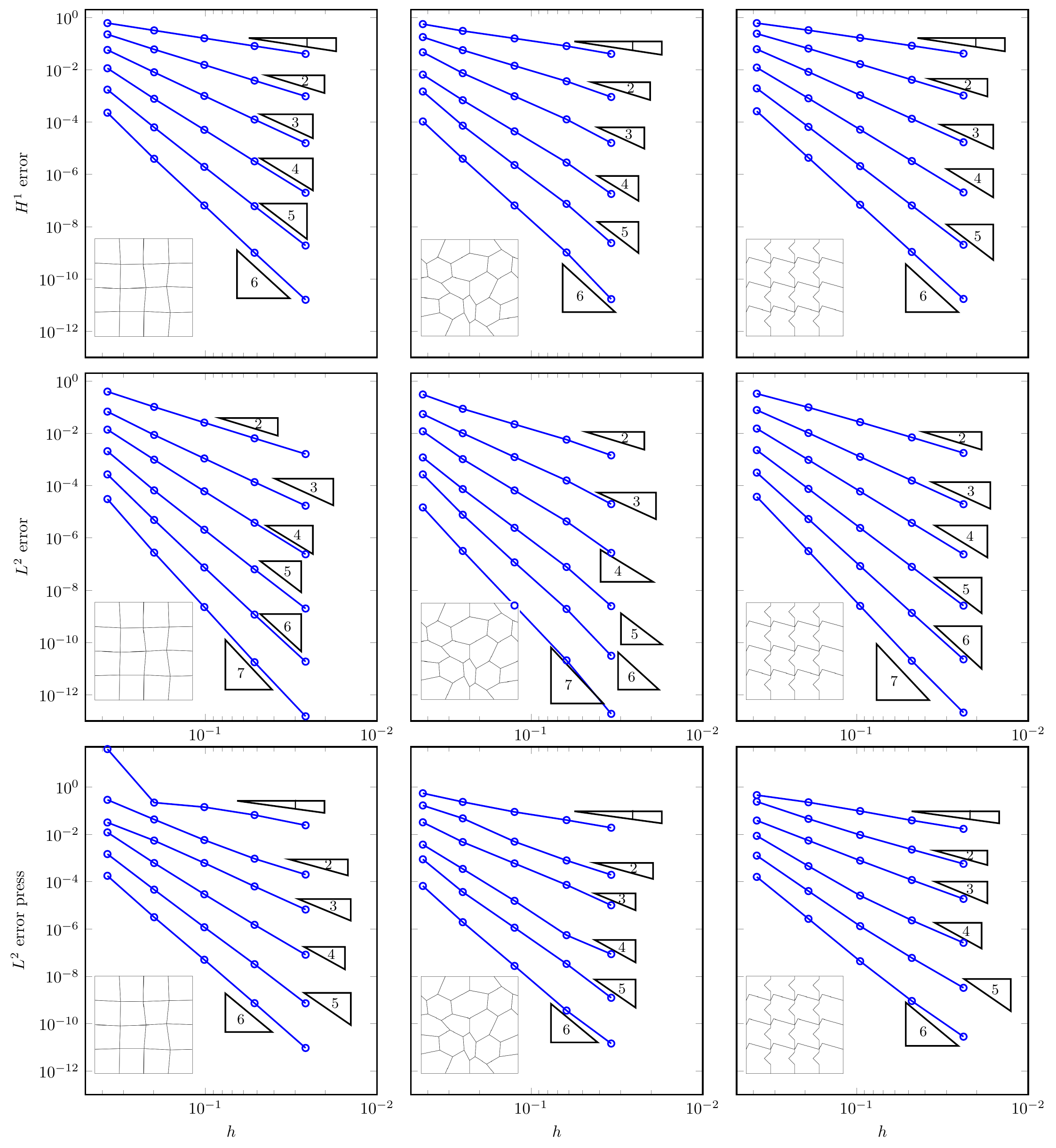} 
  \caption{Convergence curves versus the mesh size parameter $\hh$ for
    the velocity approximation measured using the energy
    norm~\eqref{eq:error:H1:velocity} (top panels) and the
    $\LTWO$-norm~\eqref{eq:error:L2:velocity} (mid panels), and for
    the pressure approximation measured using the
    $\LTWO$-norm~\eqref{eq:error:L2:pressure} (bottom panels).
    Blue lines with circles represent the error curves for the Basic
    Formulation using the enhanced virtual element
    space~\eqref{eq:BF:enhanced-space:def}.
    The right-hand side is approximated by using the projection
    operator $\Piz{k}$.
    The mesh families used in each calculations are shown in the left
    corner of each panel.
    The expected convergence slopes and rates are shown by the
    triangles and corresponding numeric labels.}
  \label{fig:h_errorPi0k}
\end{figure}

\begin{figure}
  \includegraphics[width=\textwidth,clip=]{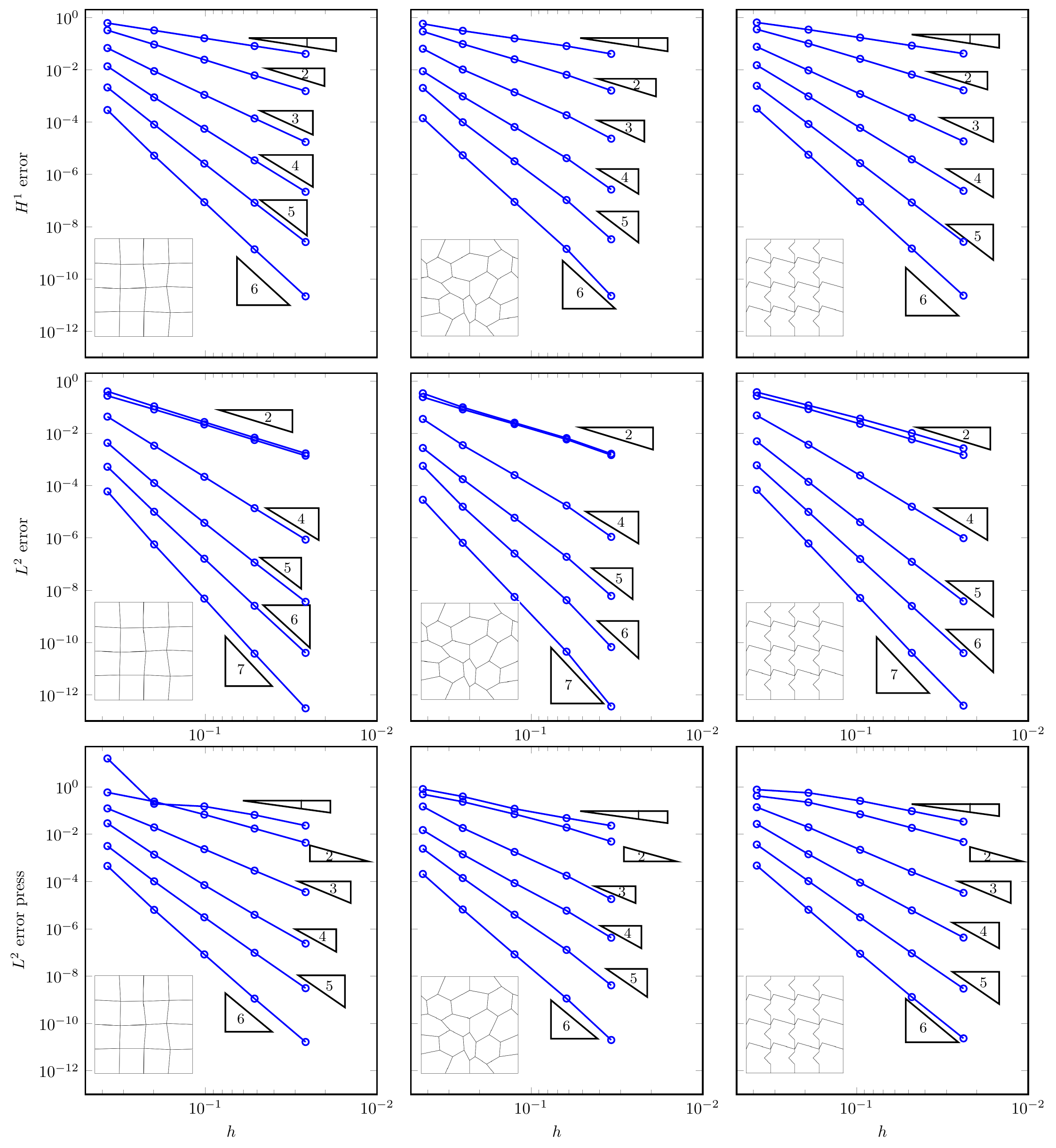} 
  \caption{Convergence curves versus the mesh size parameter $\hh$ for
    the velocity approximation measured using the energy
    norm~\eqref{eq:error:H1:velocity} (top panels) and the
    $\LTWO$-norm~\eqref{eq:error:L2:velocity} (mid panels), and for
    the pressure approximation measured using the
    $\LTWO$-norm~\eqref{eq:error:L2:pressure} (bottom panels).
    Blue lines with circles represent the error curves for the Basic
    Formulation using the virtual element
    space~\eqref{eq:BF:regular-space:def}.
    The right-hand side is approximated by using the projection
    operator $\Piz{k-2}$.
    The mesh families used in each calculations are shown in the left
    corner of each panel.
    The expected convergence slopes and rates are shown by the
    triangles and corresponding numeric labels.}
  \label{fig:h_errorPikmeno2}
\end{figure}

\subsection{The lowest-order case for $k=1$}

The case $k=1$ is critical since on triangular and square meshes this
scheme coincides with the $P^1\!-\!P^0$ Scott-Vogelius method, which
is indeed pathological.
We experimentally investigate this issue by using the mesh families
$(d)$, $(e)$, and $(f)$.
Following~\cite{Qin:1994-Thesis} and using orthonormal polynomials, a
relationship between the inf-sup constant $\beta$ and the minimum
non-zero eigenvalue of the matrix $ \matB\matA^{-1}\matB^T$ is given
by
\begin{align*}
  \beta = \sqrt{\lambda_{min}\big(\matB\matA^{-1}\matB^T\big)},
\end{align*}
where $\lambda_{min}$ is the smallest non-zero eigenvalue, and
matrices $\matB$ and $\matA$ are defined as in~\eqref{eq:saddle}.

\medskip
Note that we consider matrix $\matB$ after removing the rows and
columns that correspond to the boundary degrees of freedom.
If $l$ is the size of the kernel of matrix $\matB$, the rank of
$\matB$ is equal to $m=n_{\kb}-l$, and matrix $\matB\matA^{-1}\matB^T$
has $l$ eigenvalues equal to zero.
In Figure~\ref{fig:beta}, we show the values of the inf-sup constant
$\beta$ for $k=1$ versus the mesh size parameter $h$ of the diagonal,
criss-cross and square meshes.
It is evident from these plots that $\beta$ approaches zero, revealing
that the inf-sup condition is not satisfied.

Moreover, the numerical approximations obtained with $k \ge 2$ and
$\underline{k}=k-1$ give results as expected from the theory, as we
can see from Figure~\ref{fig:h_errorPi0ktria}.

\subsection{The case for $k=2$ and $\underline{k}=0$}
If we use the formulation with $k=2$ for the velocity virtual element
space and $\underline{k}=0$ for the pressure, we observe that the
velocity errors in the energy norm are very similar if the $\Piz{k}$
or $\Piz{k-2}$ projector is applied for the right-hand side and show a
rate of convergence of order $2$.
On the contrary, in $L^2$-norm, we note convergence of order between
$2$ and $3$ when the projector $\Piz{k}$ is applied, and convergence
of of order $2$ when the projector $\Piz{k-2}$ is applied.

For the approximation of the scalar unknown, the rate is approaching
$1$ in both cases and we observe better results with the $\Piz{k}$
projector.
The results for the case $k=2$, $\underline{k}=0$ are shown in
Figure~\ref{fig:h_errork2kk0Pi0kcompare}.

\begin{figure}
  \includegraphics[width=\textwidth,clip=]{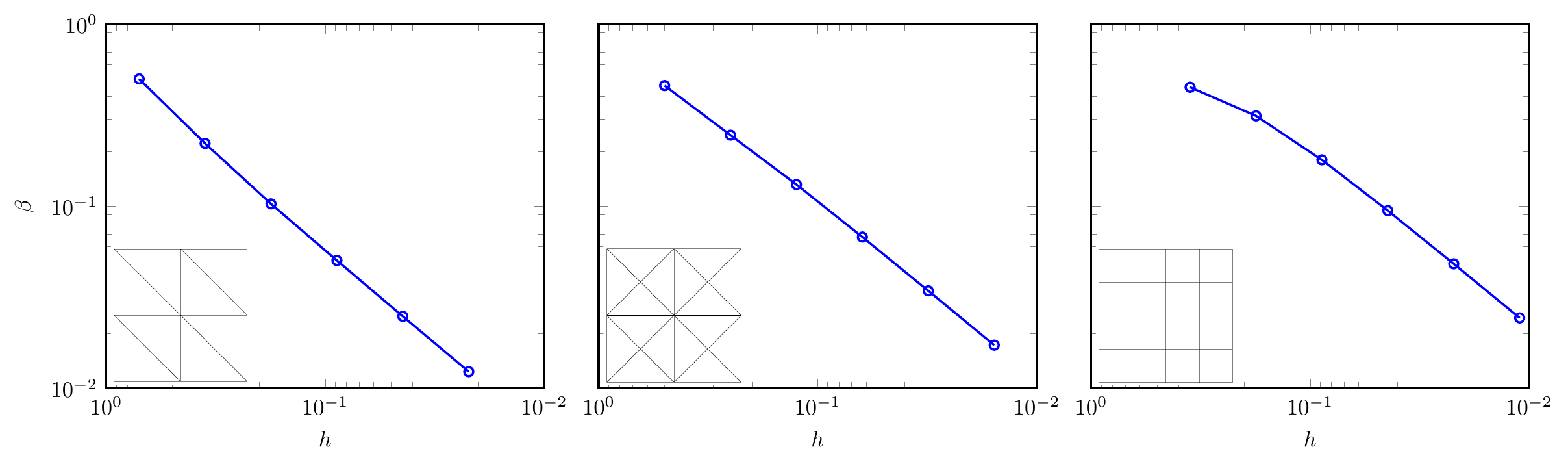} 
  \caption{Values of the inf-sup constant $\beta$ versus the mesh size parameter $\hh$.
    Blue lines with circles represent the error curves for the Basic
    Formulation using the virtual element
    space~\eqref{eq:BF:regular-space:def}.
    The mesh families used in each calculations are shown in the
    bottom-left corner of each panel.  }
  \label{fig:beta}
\end{figure}

\begin{figure}
  \includegraphics[width=\textwidth,clip=]{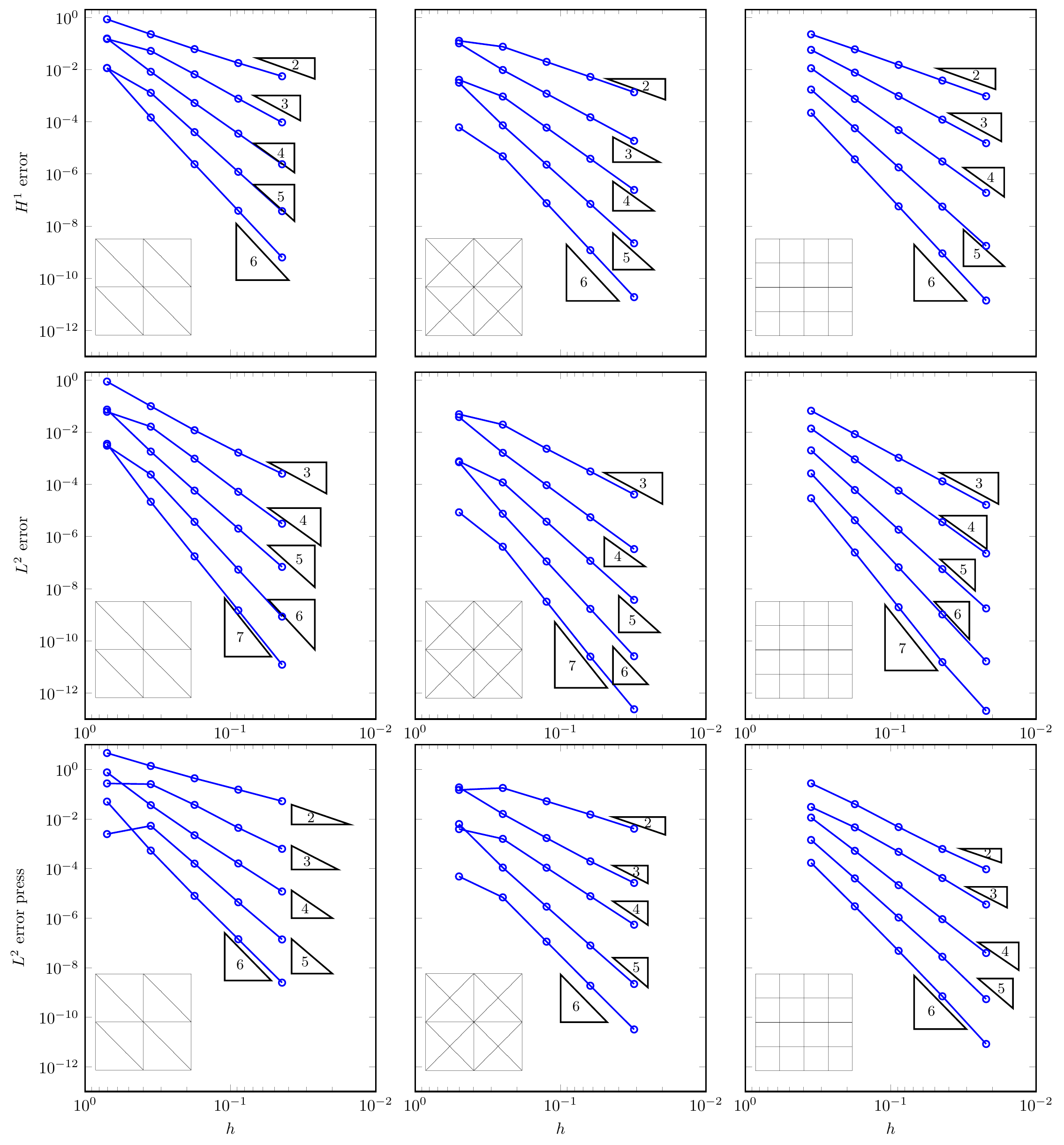} 
  \caption{Convergence curves versus the mesh size parameter $\hh$ for
    the velocity approximation measured using the energy
    norm~\eqref{eq:error:H1:velocity} (top panels) and the
    $\LTWO$-norm~\eqref{eq:error:L2:velocity} (mid panels), and for
    the pressure approximation measured using the
    $\LTWO$-norm~\eqref{eq:error:L2:pressure} (bottom panels).
    Blue lines with circles represent the error curves for the Basic
    Formulation using the enhanced virtual element
    space~\eqref{eq:BF:enhanced-space:def}.
    The right-hand side is approximated by using the projection
    operator $\Piz{k}$.
    The mesh families used in each calculations are shown in the left
    corner of each panel.
    The results with $k=1$ are not reported because there is no
    convergence.  The expected convergence slopes and rates are shown
    by the triangles and corresponding numeric labels.}
  \label{fig:h_errorPi0ktria}
\end{figure}

\begin{figure}
  \includegraphics[width=\textwidth,clip=]{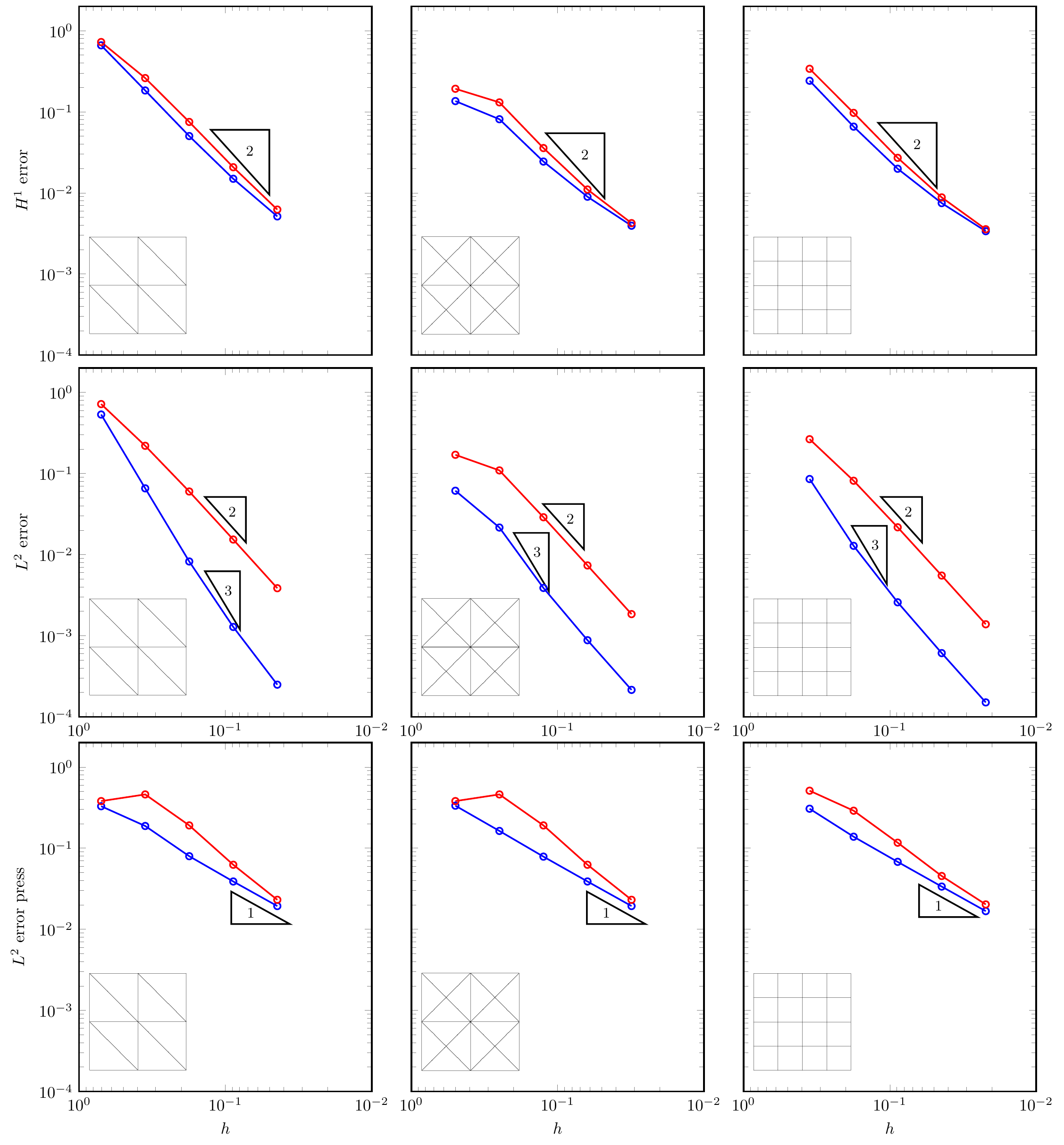} 
  \caption{Convergence curves versus the mesh size parameter $\hh$ for
    the velocity approximation measured using the energy
    norm~\eqref{eq:error:H1:velocity} (top panels) and the
    $\LTWO$-norm~\eqref{eq:error:L2:velocity} (mid panels), and for
    the pressure approximation measured using the
    $\LTWO$-norm~\eqref{eq:error:L2:pressure} (bottom panels).
    Blue lines with circles represent the error curves for the
    formulation using the enhanced virtual element
    space~\eqref{eq:BF:enhanced-space:def} with the right-hand side
    approximated by using the projection operator $\Piz{k}$.
    Red lines with circles represent the error curve with the
    right-hand side approximated by using the projection operator
    $\Piz{k-2}$.
    The mesh families used in each calculations are shown in the left
    corner of each panel.
    The convergence slopes and rates are shown by the triangles and
    corresponding numeric labels.}
  \label{fig:h_errork2kk0Pi0kcompare}
\end{figure} 

\section{Conclusions}
\label{sec:conclusions}

We studied a conforming virtual element formulation that generalizes
the Scott-Vogelius FEM for the numerical approximation of the Stokes
problem to unstructured meshes and works at any order of accuracy.
The components of the vector-valued unknown are approximated by using
the conforming regular or enhanced virtual element approximation space
that were originally introduced for the discretization of the Poisson
equation.
The scalar unknown is approximated by using discontinuous polynomials.
The stiffness bilinear form is approximated by using the orthogonal
polynomial projection of the gradients onto vector polynomials of
degree $k-1$ and adding a suitable stabilization term.
The zero divergence constraint is taken into account by projecting the
divergence equation onto the space of polynomials of degree $k-1$.
We presented a number of numerical experiments to demonstrate that
this formulation is inf-sup stable and convergent with optimal
convergence rates except for the lowest-order case (e.g., for the
polynomial order $k=1$) on triangular meshes, which corresponds to the
well-known $\PS{1}-\PS{0}$ unstable case of the Scott-Vogelius method,
and squares meshes.
Moreover, our numerical experiments show that the divergence
constraint is satisfied at the machine precision level by the
orthogonal polynomial projection of the divergence of the approximate
velocity vector.


\section*{Acknowledgments}
Dr. G. Manzini was supported by the LDRD-ER program of Los Alamos
National Laboratory under project number 20180428ER.
Los Alamos National Laboratory is operated by Triad National Security,
LLC, for the National Nuclear Security Administration of
U.S. Department of Energy (Contract No. 89233218CNA000001).
  




\clearpage

\newtheorem{applemma}[section]{Lemma}

\appendix

\renewcommand{\theequation}{\thesection.\arabic{equation}}
\setcounter{equation}{0}

\section{Approximation and orthogonality results for the virtual element method}

In this appendix, we report three technical lemmas that we use in the
convergence analysis of Section~\ref{sec4:convergence}.
The first two lemmas state basic results for the approximation of a
vector-valued field by its virtual element interpolant and a scalar
function by its polynomial projection onto the subspace of
polynomials, and are presented without a proof.
The interpolants $\vvI$ and $\qsI$ are an approximation of $\vv$ and
$\qs$, and the interpolation error can be bounded as stated in the
following lemma.
The third lemma presents an identity that is used in the proof of
Theorem~\ref{theorem:H1:abstract}.

\medskip
\begin{lemma}[Projection error~\cite{BeiraodaVeiga-Brezzi-Cangiani-Manzini-Marini-Russo:2013,Brenner-Scott:1994}]
  \label{lemma:projection:error}
  Under Assumptions~\textbf{(M1)}-\textbf{(M2)}, for every
  vector-valued field $\vv\in\big[\HS{s+1}(\P)\big]^2$ and scalar
  function $\qs\in\HS{s}(\P)$ with $1\leq\ss\leq\ell$, there exists a
  vector polynomial $\vv_{\pi}\in\big[\PS{\ell}(\P)\big]^2$ and a
  scalar polynomial $\qs_{\pi}\in\PS{\ell}(\P)$ such that
  \begin{align}
    &\norm{\vv-\vv_{\pi}}{0,\P} + \hP\snorm{\vv-\vv_{\pi}}{1,\P}\leq\Cs\hP^{s+1}\snorm{\vv}{s+1,\P},\\[0.75em]
    &\norm{\qs-\qs_{\pi}}{0,\P} + \hP\snorm{\qs-\qs_{\pi}}{1,\P}\leq\Cs\hP^{s} \snorm{\qs}{s,\P}
  \end{align}
  for some positive constant $\Cs$ that is independent of $\hP$ but
  may depend on the polynomial degree $\ell$ and the mesh regularity
  constant $\varrho$.
\end{lemma}

\medskip
\begin{lemma}[Interpolation error~\cite{BeiraodaVeiga-Brezzi-Cangiani-Manzini-Marini-Russo:2013,Brenner-Scott:1994}]
  \label{lemma:interpolation:error}
  Under Assumptions~\textbf{(M1)}-\textbf{(M2)}, for every
  vector-valued field $\vv\in\big[\HS{s+1}(\P)\big]^2$ and scalar
  function $\qs\in\HS{s}(\P)$ with $1\leq\ss\leq\ell$, there exists a
  vector valued-field $\vvI\in\Vvhk(\P)$ and a scalar field
  $\qs\in\Qsh_{\ell}$ such that
  \begin{align}
    \norm{\vv-\vvI}{0,\P} + \hP\snorm{\vv-\vvI}{1,\P}\leq\Cs\hP^{s+1}\snorm{\vv}{s+1,\P},\\[0.5em]
    \norm{\qs-\qsI}{0,\P} + \hP\snorm{\qs-\qsI}{1,\P}\leq\Cs\hP^{s} \snorm{\qs}{s,\P}
  \end{align}
  for some positive constant $\Cs$ that is independent of $\hP$ but
  may depend on the polynomial degree $\ell$ and the mesh regularity
  constant $\varrho$.
\end{lemma}

\medskip
\begin{lemma}[Orthogonality between $\DIV(\uvh-\uvI)$ and $\psh-\psI$]
  Let $\uv\in\big[\HONEzr(\Omega)\big]^2$ be the exact solution of the
  variational formulation of the Stokes problem given
  in~\eqref{eq:stokes:var:A}-\eqref{eq:stokes:var:B}.
  Let $(\uvh,\psh)\in\Vvhk\times\Qshkk$ be the solution of the virtual
  element
  approximation~\eqref{eq:stokes:vem:A}-\eqref{eq:stokes:vem:B}.
  Then, it holds that
  \begin{align}
    \bs(\uvh-\uvI,\psh-\psI) &= 0 \qquad\forall\qs\in\Qsh.
    \label{eq:aux:20}
  \end{align}
\end{lemma}
\begin{proof}
  First, note that the composed operator $\PizP{k-1}\DIV(\cdot)$ only
  depends on the degrees of freedom of its argument.
  These degrees of freedom are the same for $\uv$ and
  its virtual element interpolation $\uvI$, so that it must hold
  that $\PizP{k-1}\DIV\uv=\PizP{k-1}\DIV\uvI$.
  Using this property and the definition of the orthogonal projection
  $\PizP{k-1}$ yield
  \begin{align*}
    \bs(\uv,\qs)
    &= \sum_{\P}\int_{\P}\qs\,\DIV\uv\dV
    = \sum_{\P}\int_{\P}\qs\PizP{k-1}\DIV\uv\dV
    = \sum_{\P}\int_{\P}\qs\PizP{k-1}\DIV\uvI\dV
    \\[0.5em]
    &= \sum_{\P}\int_{\P}\qs\,\DIV\uvI\dV
    = \bs(\uvI,\qs)
  \end{align*}
  for every $\qs\in\PS{k-1}(\P)$.
  Eventually, we note that $0=\bs(\uv,\qs)=\bs(\uvI,\qs)$ from
  Eq.~\eqref{eq:stokes:var:B} and $\bs(\uvh,\qs)=\bsh(\uvh,\qs)$
  from Eq.~\eqref{eq:bsh-bs}.
  So, it holds that $\bs(\uvh-\uvI,\qs)=0$, and relation~\eqref{eq:aux:20}
  immediately follows by setting $\qs=\psh-\psI$.
\end{proof}

\end{document}